\documentclass[12pt]{amsart}

\usepackage{amssymb,amsfonts}
\usepackage{amsthm}
\usepackage[dvipsnames]{xcolor}
\usepackage{graphicx}
\usepackage[all,arc]{xy}
\usepackage[colorlinks=true, allcolors=blue]{hyperref}
\usepackage{enumerate}
\usepackage{mathrsfs}
\usepackage{tikz-cd}
\usepackage{import}
\usepackage[utf8]{inputenc}
\usepackage{hyperref}
\usepackage[left=1in,top=1in,right=1in,bottom=1in]{geometry}
\usepackage{mathtools}

\usetikzlibrary{decorations.pathreplacing,calligraphy,decorations,decorations.markings,arrows}

\newtheorem{thm}{Theorem}[section]
\newtheorem{cor}[thm]{Corollary}
\newtheorem{prop}[thm]{Proposition}
\newtheorem{lem}[thm]{Lemma}

\theoremstyle{definition}
\newtheorem{defn}[thm]{Definition}

\newtheorem{exmp}[thm]{Example}

\newtheorem{rmk}[thm]{Remark}

\newcommand{\Out}{\operatorname{Out}}
\newcommand{\Aut}{\operatorname{Aut}}
\newcommand{\sgn}{\mathrm{sgn}}

\newcommand{\Conf}{\operatorname{Conf}}
\newcommand{\into}{\hookrightarrow}

\newcommand{\Gp}{\mathcal{G}}               
\newcommand{\Q}{\mathbb{Q}}
\newcommand{\R}{\mathbb{R}}

\DeclareMathOperator{\st}{st}
\newcommand{\ol}[1]{\overline{#1}}
\newcommand{\F}{\mathcal{F}}
\DeclareMathOperator{\CV}{CV}

\newcommand{\ul}[1]{\underline{#1}}

\newcommand{\Ho}{\mathrm{H}}
\newcommand{\Hc}{\Ho_c}
\newcommand{\GC}{\mathrm{GC}}
\newcommand{\GH}{\mathrm{GH}}
\newcommand{\rH}{\widetilde{\Ho}}
\DeclareMathOperator{\stab}{stab}

\newcommand{\A}{\mathcal{A}}

\newcommand{\U}{\mathcal{U}}

\newcommand{\G}{{G}}                

\renewcommand{\P}{\mathcal{P}}

\newcommand{\Top}{\mathsf{Top}}

\DeclareMathOperator{\diag}{diag}

\newcommand{\isoto}{\xrightarrow{\sim}}

\DeclareMathOperator{\inv}{inv}

\DeclareMathOperator{\im}{im}
\DeclareMathOperator{\Vect}{Vect}
\DeclareMathOperator{\GrphCat}{Grph}
\DeclareMathOperator{\Iso}{Iso}
\newcommand{\V}{\mathbb{V}}
\newcommand{\col}{\colon}
\newcommand{\Specht}[1]{\chi_{#1}}
\newcommand{\op}{\mathrm{op}}
\newcommand{\cA}{\mathcal{A}}
\newcommand{\cF}{\mathcal{F}}
\newcommand{\cM}{\mathcal{M}}
\newcommand{\cU}{\mathcal{U}}

\newcommand{\can}{
\begin{tikzpicture}[scale=.25,baseline=0pt]
\tikzstyle{every node}=
[draw,circle,fill=black,minimum size=.5pt,inner sep=0pt];
\node (bl) at (0,0) {};
\node (br) at (0,.7) {};
\node (tl) at (.7,0) {};
\node (tr) at (.7,.7) {};
\draw[-] (tl) edge[bend right=30] (tr);
\draw[-] (tl) edge[bend left=30] (tr);
\draw[-] (bl) edge[bend right=30] (br);
\draw[-] (bl) edge[bend left=30] (br);
\draw[-] (bl) edge (tl);
\draw[-] (br) edge (tr);
\end{tikzpicture}
}
\newcommand{\goggles}{
\begin{tikzpicture}[scale=.25,baseline=-1pt]
\tikzstyle{every node}=
[draw,circle,fill=black,minimum size=.5pt,inner sep=0pt];
\node (l) at (0,0) {};
\node (c) at (0.5,0.8) {};
\node (r) at (1,0) {};
\draw[-] (l)--(r);
\draw[-] (l) edge[bend right=25] (c);
\draw[-] (r) edge[bend right=25] (c);
\draw[-] (l) edge[bend left=25] (c);
\draw[-] (r) edge[bend left=25] (c);
\end{tikzpicture}
}
\newcommand{\banana}{
\begin{tikzpicture}[scale=.3]
\tikzstyle{every node}=
[draw,circle,fill=black,minimum size=.5pt,inner sep=0pt];
\node (0) at (0,0) {};
\node (1) at (0,.7) {};
\draw[-] (0) edge[bend right=85, min distance=5mm] (1);
\draw[-] (0) edge[bend right=45] (1);
\draw[-] (0) edge[bend left=45] (1);
\draw[-] (0) edge[bend left=85, min distance=5mm] (1);
\end{tikzpicture}
}
\newcommand{\kfour}{
\begin{tikzpicture}[scale=.2]
\tikzstyle{every node}=
[draw,circle,fill=black,minimum size=.5pt,inner sep=0pt];
\node (bl) at (0,0) {};
\node (br) at (0.606,-0.35) {};
\node (tl) at (-0.606,-0.35) {};
\node (tr) at (0,0.7) {};
\draw[-] (bl)--(br)--(tr)--(tl)--(bl)--(tr);
\draw[-] (tl)--(br);
\end{tikzpicture}
}

\newcommand{\bigcan}{
\begin{tikzpicture}[scale=.9]
\tikzstyle{every node}=
[draw,circle,fill=black,minimum size=2pt,inner sep=0pt];
\node (bl) at (0,0) {};
\node (br) at (0,.7) {};
\node (tl) at (.7,0) {};
\node (tr) at (.7,.7) {};
\tikzstyle{every node}=[scale=.7];
\draw[-] (tl) edge[bend right=30] (tr);
\draw[-] (tl) edge[bend left=30] (tr);
\draw[-] (bl) edge[bend right=30] (br);
\draw[-] (bl) edge[bend left=30] (br);
\draw[-] (bl) edge (tl);
\draw[-] (br) edge (tr);
\end{tikzpicture}
}
\newcommand{\biggoggles}{
\begin{tikzpicture}[scale=.8]
\tikzstyle{every node}=
[draw,circle,fill=black,minimum size=2pt,inner sep=0pt];
\node (l) at (0,0) {};
\node (c) at (0.5,0.8) {};
\node (r) at (1,0) {};
\tikzstyle{every node}=[scale=.7];
\draw[-] (l)--  (r);
\draw[-] (l) edge[bend right=25] (c);
\draw[-] (r) edge[bend right=25] (c);
\draw[-] (l) edge[bend left=25] (c);
\draw[-] (r) edge[bend left=25] (c);
\end{tikzpicture}
}
\newcommand{\bigbanana}{
\begin{tikzpicture}[scale=.9]
\tikzstyle{every node}=
[draw,circle,fill=black,minimum size=2pt,inner sep=0pt];
\node (0) at (0,0) {};
\node (1) at (0,.7) {};
\draw[-] (0) edge[bend right=85, min distance=5mm] (1);
\draw[-] (0) edge[bend right=45] (1);
\draw[-] (0) edge[bend left=45] (1);
\draw[-] (0) edge[bend left=85, min distance=5mm] (1);
\end{tikzpicture}
}
\newcommand{\bigkfour}{
\begin{tikzpicture}[scale=.65]
\tikzstyle{every node}=
[draw,circle,fill=black,minimum size=2pt,inner sep=0pt];
\node (c) at (0,0) {};
\node (r) at (0.606,-0.35) {};
\node (l) at (-0.606,-0.35) {};
\node (t) at (0,0.7) {};
\draw[-] (c)--(t)--(r)--(l);
\draw[-] (r)--(c)--(l)--(t);
\end{tikzpicture}
}

\title{A Serre spectral sequence for the moduli space of tropical curves}
\author{Christin Bibby}
\address{Department of Mathematics, Louisiana State University, Baton Rouge, LA 70803}
\email{\url{bibby@math.lsu.edu}}
\author{Melody Chan}
\address{Department of Mathematics, Brown University, Box 1917, Providence, RI 02912}
\email{\url{melody_chan@brown.edu}}
\author{Nir Gadish}
\address{Department of Mathematics, University of Michigan, Ann Arbor, MI}
\email{\url{gadish@umich.edu}}
\author{Claudia He Yun}
\address{Department of Mathematics, University of Michigan, Ann Arbor, MI}
\email{\url{clyun@umich.edu}}



\begin{document}

\begin{abstract}
We construct, for all $g\ge 2$ and $n\ge 0$, a spectral sequence of rational $S_n$-representations which computes the $S_n$-equivariant rational cohomology of the tropical moduli spaces of curves $\Delta_{g,n}$ in terms of compactly supported cohomology groups of configuration spaces of $n$ points on graphs of genus $g$.  
Using the canonical $S_n$-equivariant isomorphisms $\widetilde{H}^{i-1}(\Delta_{g,n};\Q) \cong W_0 H^i_c(\cM_{g,n};\Q)$, we calculate the weight $0$, compactly supported rational cohomology of the moduli spaces $\cM_{g,n}$ in the range $g=3$ and $n\le 9$, with partial computations available for $n\le 13$.
\end{abstract}
	
\maketitle

\section{Introduction}

We start with some very brief recollections and motivation.  Let $\cM_{g,n}$ denote the complex moduli stack of  smooth curves of genus $g$ with $n$ distinct marked points. The map $\cM_{g,n}\to \cM_g = \cM_{g,0}$ forgetting marked points is a fibration, with fiber isomorphic to $\Conf_n(S_g)$, the configuration space of $n$ distinct points on a genus $g$ surface.   The associated Serre spectral sequence on compactly supported cohomology is
\begin{equation}\label{eq:algebraic serre ss}
    E_2^{p,q} = \Hc^p(\cM_g, \Hc^q(\Conf_n(S_g)) \Rightarrow \Hc^{p+q}(\cM_{g,n}).
\end{equation}

In this paper, we present a {\em tropical} analogue of this Serre spectral sequence, serving as a major tool for calculation, where the role of the $E_2$-page is played by certain graph cohomology introduced below. Then we use this ``tropical Serre spectral sequence'' to extract new calculations on the weight 0, compactly supported rational cohomology of the algebraic variety $\cM_{g,n}$ with its action by the respective symmetric group.  The new calculations we achieve are for $g=3$ in the range $n\le 9$, and multiplicities of the trivial and sign representations for $n\leq 13$. These calculations are presented in Tables~\ref{table:genus3} and~\ref{table:sign_trivial}. The case $g=2$ was taken up in our previous work \cite{bibby-chan-gadish-yun-homology} and in work of the third author with Hainaut \cite{gadish-hainaut}, when the spectral sequence below is simpler, and in particular degenerates at $E_1$.  Our spectral sequence is a geometric instantiation of an earlier spectral sequence due to Turchin--Willwacher \cite{turchin-willwacher-commutative} for commutative graph complexes with coefficients in Hochschild-Pirashvili cohomology of graphs.  We explain this connection, which requires rectifying some details, in \S\ref{sec:related}.  See \cite{bibby-chan-gadish-yun-homology} and \S\ref{sec:related}~for a further literature summary and earlier references, including to the cases of $g\le 1$.  

To state our first theorem, let $\Delta_{g,n}$ denote the moduli space of unit-volume genus $g$ stable tropical curves with $n$ distinct marked points. 
This space has the following two equivalent descriptions:
\begin{itemize}
\item the quotient of the simplicial completion of $n$-marked Outer Space by the action of $\Gamma_{g,n}$;
\item the quotient of the curve complex $\mathcal{C}_{g,n}$ by the action of the pure mapping class group $\mathrm{PMod}_{g,n}. $
\end{itemize}
Recently, $\Delta_{g,n}$ has been rechristened as a tropical moduli space, emphasizing its role as the boundary complex of the Deligne-Mumford-Knudsen compactification $\cM_{g,n}\subset \overline{\cM}_{g,n}$ (see \cite{acp}). 
If none of these descriptions is familiar, then the reader may roughly picture $\Delta_{g,n}$ as a simplicial complex whose geometric realization parametrizes graphs of genus (first Betti number) $g$ equipped with $n$ distinct marked points on them.

Our tropical analogue of the Serre spectral sequence features the following \textit{graph complex}. Let $\Gamma_g^{(2)}$ be the category of stable $2$-connected graphs of genus $g$: here, stable means that each vertex has valence at least $3$, and $2$-connected means that no vertex separates. Morphisms in $\Gamma_g^{(2)}$ are generated by edge contractions and isomorphisms. Consider $\Out(F_g)$, the outer automorphism group of the free group of rank $g$, i.e., automorphisms of $F_g$ up to conjugation. As in~\cite{turchin-willwacher-hochschild} (see  \S\ref{sec:graph complexes}), every $\Out(F_g)$--representation $V$ defines a graph complex
\begin{equation}
    \GC^p(\Gamma_g^{(2)}; V) \cong \bigoplus_{\substack{[\G] \in \operatorname{Iso}(\Gamma_g^{(2)})\\ |E(\G)| = p+1}} \left( V \otimes \det(E(\G)) \right)^{\Aut(\G)},
\end{equation}
where $E(\G)$ denotes the set of edges in the graph $\G$ and $\det(E(\G))$ is the one-dimensional vector space $\Lambda^{|E(\G)|}(\Q^{E(\G)})$. See \S\ref{sec:graph complexes} for the differential and the action of $\Aut(\G)$ on $V$.

The relevant examples of $\Out(F_g)$-representations for the next theorem are the compactly supported cohomologies $\Hc^q(\Conf_n(R_g);\Q)$, where $R_g$ is the rose graph with one vertex and $g$ loops, and which are nonzero only if $q=n-1$ or $q=n$ (see \cite{bibby-chan-gadish-yun-homology}). The group $\Out(F_g)$ coincides with the group of homotopy equivalences $R_g\to R_g$ up to homotopy, and these act on $\Hc^q(\Conf_n(R_g);\Q)$ diagonally as observed in \cite[\S2]{bibby-chan-gadish-yun-homology}.

\begin{thm}\label{thm:ss}\label{thm:main}
For any $g\geq 2$ and $n\ge 1$,
there is a spectral sequence of rational $S_n$-representations with $E_1$-page 
having every row given by a graph complex
\begin{equation}\label{eq:the-ss} E_1^{p,q} = \GC^p(\Gamma_g^{(2)}; \Hc^q(\Conf_n(R_g);\Q))
\Longrightarrow
\widetilde{\Ho}^{p+q}(\Delta_{g,n};\Q).\end{equation}
This $E_1$-page
is supported  in rows $q=n-1$ and $q=n$.  
The spectral sequence
degenerates at $E_1$ when $g=2$; at $E_2$ when $g=3$; and at $E_3$ when $g>3$.
\end{thm}

The geometric fact relating our theorem to Serre's spectral sequence is the following, which we discuss in \S\ref{sec:graph complexes}.
An $\Out(F_g)$-representation $V$ defines a local system $\V$ on the topological quotient stack ${\Delta}_{g}^{(2)}$ of $2$-connected tropical curves of first Betti number $g$, and
\begin{equation}
    \Hc^p({\Delta}_{g}^{(2)} ; \V) \cong \Ho^p\big( \GC^*(\Gamma_g^{(2)}; V)\big).
\end{equation}
With this one can think of Theorem \ref{thm:ss} as describing a spectral sequence with
\begin{equation}\label{eq:serre ss}
    E_2^{p,q} \cong \Hc^p(\Delta_g^{(2)};\Hc^q(\Conf_n(\G);\Q))\Longrightarrow
\widetilde{\Ho}^{p+q}(\Delta_{g,n};\Q)
\end{equation}
analogous to the algebraic one for $\cM_{g,n}$ in \eqref{eq:algebraic serre ss}.

The relatively simple form of this spectral sequence
may surprise experts who are familiar with moduli spaces of tropical curves. As discussed in \S\ref{Mtrop}, there is a more natural spectral sequence associated to forgetting marked points, but its form is far more complicated, does not feature configuration spaces, and is generally hard to use in computer calculations.

\subsection{Calculations}
Using Theorem~\ref{thm:main}, we calculated all the weight $0$ cohomology groups with compact support \[\mathrm{Gr}_{0}^W {\Ho}_c^*(\cM_{3,n};\Q)\] for $n\le 9$, as $S_n$-representations.  These calculations are displayed in Table~\ref{table:genus3}.  A key result for achieving these computations is  Proposition \ref{prop:differential special edge}, a statement about the differential of the graph complex that simplifies calculations in two ways.
First, it implies the desired degeneration of the spectral sequence, so that there is no $d_2$ differential when $g=3$.
Second, it allows us to prune the graph complex, meaning that the matrices needed for computation can be made significantly smaller.
Examples \ref{ex:genus3 injective diff} and \ref{ex:genus3 surjective diff} explain how this is used to compute the $S_n$-character of the degree $n+3$ and $n+5$ homology of $\Delta_{3,n}$.
The only other nontrivial homology group is in degree $n+4$, and the equivariant Euler characteristic from \cite[Theorem 1.1]{cfgp-sn} is then used to compute these characters.

Taking one step back to the case of genus $g=2$, the only graph in $\Gamma_2^{(2)}$ up to isomorphism is the theta graph $\Theta$, which has two vertices and three edges between them.
The tropical Serre spectral sequence reduces to a single nontrivial column on the $E_1$ page. This yields the following expression, proven more directly in \cite[Theorem 3.2]{bibby-chan-gadish-yun-homology}:
\[\widetilde{\Ho}^{n+2-*}(\Delta_{2,n};\Q)\cong E_1^{2,n-*}\cong \left[\Hc^{n-*}(\Conf_n(\Theta);\Q)\otimes\det(E(\Theta))\right]^{\Aut\Theta}.\]
Data generated using this formula are included in Tables 1 and 2 of op. cit., which provide the full $S_n$-character for $n\leq 11$ and partial data for $n\leq 17$. Furthermore, \cite{gadish-hainaut} use the formula to exhibit super-exponentially growing families of subrepresentations of $\widetilde{\Ho}^{n+2-*}(\Delta_{2,n};\Q)$ with a suggestively geometric description.

Our calculations in genus $g=2$ and $g=3$ were implemented in SageMath, 
using some further reductions from \cite{bibby-chan-gadish-yun-homology} and Appendix \ref{appendix}.
All our code and data 
are presented at \href{https://github.com/ClaudiaHeYun/BCGY}{this URL}\footnote{\texttt{https://github.com/ClaudiaHeYun/BCGY}}.

\begin{rmk}
When $g\ge 4$, we have no reason to believe that the $d_2$ differential of~\eqref{eq:the-ss} is trivial, as it is for $g=2$ and $3$. This is the most immediate obstruction to further computations.  However, the fact that $H^q_c(\Conf_n(R_g);\Q)$ is nonzero only in two degrees $q=n-1$ and $q=n$ implies immediately that all further differentials vanish in the spectral sequence in Theorem~\ref{thm:ss}, and hence there is convergence at $E_3$ if not sooner.

We also note the growth of the number of isomorphism classes of graphs in $\Gamma_g^{(2)}$.  There is 1 such for $g=2$ and 4 such for $g=3$, but already 17 graphs for $g=4$.  See Figure~\ref{fig:genus3graphs}.  Without our reduction to $2$-connected graphs, there would be 379 such graphs for $g=4$.
\end{rmk}

\begin{figure}[htb]
\begin{tikzpicture}
\node (c) at (2,1) {$\bigcan$};
\node (k) at (2,-1) {$\bigkfour$};
\node (g) at (0,0) {$\biggoggles$};
\node (b) at (-2,0) {$\bigbanana$};
\draw[->] (k) -- (g);
\draw[->] (g) -- (b);
\draw[->] (1.45,1) -- (g);
\end{tikzpicture}
\caption{The four non-isomorphic graphs in $\Gamma_3^{(2)}$.  Arrows indicate the existence of a single edge contraction. }
\label{fig:genus3graphs}
\end{figure}
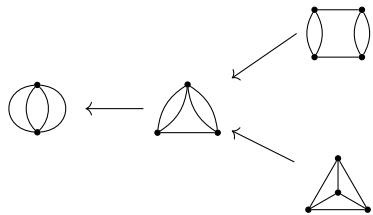

\renewcommand{\arraystretch}{1.2}
\begin{table}[hbt]
{\footnotesize
\begin{tabular}{|r|r|p{.9\textwidth}|}
\hline
$n$ & $i$ & $S_n$-character of $\rH^{n+5-i}(\Delta_{3,n};\Q) \cong \mathrm{Gr}^W_{0} \Hc^{n+6-i}(\cM_{3,n};\Q)$ \\
\hline
\hline
 & 0 & 0 \\
 \cline{2-3}
1 & 1 & $\Specht{(1)}$ \\
 \cline{2-3}
 & 2 & 0 \\
\hline
 & 0 & 0\\
 \cline{2-3}
2  & 1 & 0\\
 \cline{2-3}
  & 2 & 0\\
\hline
 & 0 & $\Specht{(3)}$\\
 \cline{2-3}
3  & 1 & 0\\
 \cline{2-3}
  & 2 & 0\\
\hline
 & 0 & $\Specht{(2,2)}$\\
 \cline{2-3}
4  & 1 & $\Specht{(2,1,1)}$\\
 \cline{2-3}
  & 2 & 0\\
\hline
 & 0 & $2\Specht{(3,1,1)}+\Specht{(2,1,1,1)}$\\
 \cline{2-3}
5  & 1 & $\Specht{(5)}+\Specht{(4,1)}+\Specht{(3,2)}+ \Specht{(3,1,1)}$\\
 \cline{2-3}
  & 2 & 0\\
\hline
 & 0 & $\Specht{(5,1)}+\Specht{(4,2)}+3\Specht{(4,1,1)}+3\Specht{(3,3)}+\Specht{(3,2,1)}+\Specht{(3,1,1,1)}+\Specht{(2,2,1,1)}$ \\
 \cline{2-3}
6  & 1 & $\Specht{(6)}+\Specht{(5,1)}+3\Specht{(4,2)}+\Specht{(3,3)}+2\Specht{(3,2,1)}+\Specht{(2,2,2)}+\Specht{(2,2,1,1)} + 2\Specht{(2,1,1,1,1)}$\\
 \cline{2-3}
  & 2 & 0\\
\hline
 & 0 & $\Specht{(7)}+2\Specht{(6,1)}+5\Specht{(5,2)}+2\Specht{(4,3)}+5\Specht{(4,2,1)}+2\Specht{(3^2,1)}+5\Specht{(3,2^2)}+3\Specht{(3,2,1^2)}+3\Specht{(3,1^4)}+2\Specht{(2^3,1)}+2\Specht{(2^2,1^3)}+2\Specht{(2,1^5)}$\\
 \cline{2-3}
7 & 1 & $2\Specht{(5,2)}+2\Specht{(5,1^2)}+\Specht{(4,3)}+5\Specht{(4,2,1)}+5\Specht{(4,1^3)}+2\Specht{(3^2,1)}+3\Specht{(3,2^2)}+4\Specht{(3,2,1^2)}+3\Specht{(3,1^4)}+\Specht{(2^3,1)}$ \\
 \cline{2-3}
  & 2 & 0 \\ 
\hline
 & 0 & $\Specht{(8)}+4\Specht{(6,2)}+2\Specht{(6,1^2)}+2\Specht{(5,3)}+9\Specht{(5,2,1)}+6\Specht{(5,1^3)}+4\Specht{(4^2)}+8\Specht{(4,3,1)}+12\Specht{(4,2^2)}+12\Specht{(4,2,1^2)}+8\Specht{(4,1^4)}+2\Specht{(3^2,2)}+9\Specht{(3^2,1^2)}+8\Specht{(3,2^2,1)}+8\Specht{(3,2,1^3)}+3\Specht{(3,1^5)}+4\Specht{(2^4)}+2\Specht{(2^2,1^4)}$\\
 \cline{2-3}
8 & 1 & $8\Specht{(6,1^2)}+2\Specht{(5,3)}+9\Specht{(5,2,1)}+8\Specht{(5,1^3)}+\Specht{(4^2)}+13\Specht{(4,3,1)}+2\Specht{(4,2^2)}+13\Specht{(4,2,1^2)}+\Specht{(4,1^4)}+5\Specht{(3^2,2)}+11\Specht{(3^2,1^2)}+7\Specht{(3,2^2,1)}+8\Specht{(3,2,1^3)}+\Specht{(2^4)}+6\Specht{(2^3,1^2)}+4\Specht{(2^2,1^4)}+3\Specht{(2,1^7)}+\Specht{(1^8)}$ \\
 \cline{2-3}
  & 2 & $\Specht{(8)}$ \\
\hline
& 0 & 
$\Specht{(1^9)} + 3\Specht{(2, 1^8)} + 5\Specht{(2, 2, 1^5)} + 11\Specht{(2^3, 1^3)} + 5\Specht{(2^4, 1)} + 5\Specht{(3, 1^6)} + 12\Specht{(3, 2, 1^4)} + 26\Specht{(3, 2^2, 1^2)} + 7\Specht{(3, 2^3)} + 17\Specht{(3^2, 1^3)} + 24\Specht{(3^2, 2, 1)} + 12\Specht{(3^3)} + 3\Specht{(4, 1^5)} + 25\Specht{(4, 2, 1^3)} + 25\Specht{(4, 2^2, 1)} + 36\Specht{(4, 3, 1^2)} + 22\Specht{(4, 3, 2)} + 13\Specht{(4^2, 1)} + 6\Specht{(5, 1^4)} + 34\Specht{(5, 2, 1^2)} + 9\Specht{(5, 2^2)} + 29\Specht{(5, 3, 1)} + 2\Specht{(5, 4)} + 15\Specht{(6, 1^3)} + 16\Specht{(6, 2, 1)} + 6\Specht{(6, 3)} + 9\Specht{(7, 1^2)}$
\\
\cline{2-3}
9 & 1 & 
$3\Specht{(2^3, 1^3)} + 3\Specht{(2^4, 1)} + 8\Specht{(3, 1^6)} + 15\Specht{(3, 2, 1^4)} + 26\Specht{(3, 2^2, 1^2)} + 10\Specht{(3, 2^3)} + 9\Specht{(3^2, 1^3)} + 22\Specht{(3^2, 2, 1)} + 9\Specht{(3^3)} + 12\Specht{(4, 1^5)} + 30\Specht{(4, 2, 1^3)} + 35\Specht{(4, 2^2, 1)} + 26\Specht{(4, 3, 1^2)} + 27\Specht{(4, 3, 2)} + 9\Specht{(4^2, 1)} + 7\Specht{(5, 1^4)} + 29\Specht{(5, 2, 1^2)} + 15\Specht{(5, 2^2)} + 29\Specht{(5, 3, 1)} + 7\Specht{(5, 4)} + 4\Specht{(6, 1^3)} + 16\Specht{(6, 2, 1)} + 11\Specht{(6, 3)} + 5\Specht{(7, 1^2)} + 6\Specht{(7, 2)} + 4\Specht{(8, 1)} + \Specht{(9)}$
\\
\cline{2-3}
& 2 & 
$\Specht{(1^9)} + \Specht{(3, 2^3)} + \Specht{(4^2, 1)} + \Specht{(5, 2^2)} + \Specht{(7, 2)}$
\\
\hline
\end{tabular}
}
\bigskip
\caption{Reduced homology of the moduli space of $n$-marked genus 3 tropical curves as a representation of the symmetric group $S_n$. The rows labeled with $i=2$ contribute to $\rH^{4g-4+n}(\cM_{3,n};\Q)$ -- the virtual cohomological dimension.}
\label{table:genus3}
\end{table}

\begin{rmk}\label{rmk-on-BV}
We refer the reader to the recent paper of Borinsky--Vermaseren \cite{borinsky-vermaseren-sn} for a closely related, but different, calculation. They study the {\em moduli spaces of graphs} $\cM\mathcal{G}_{g,n}$ for $g>0$ and $2g-2+n>0$, and specifically the $S_n$-equivariant Euler characteristics thereof.  
The spaces $\cM\mathcal{G}_{g,n}$ are homeomorphic to the open subspace $\Delta_{g,n}^{\mathrm{pure}}$ of $\Delta_{g,n}$ parametrizing tropical curves with all vertex weights zero.  

In contrast, here we are interested in the {\em compactly supported} cohomology of $\Delta_{g,n}^{\mathrm{pure}}$. Indeed, there is an isomorphism \[\widetilde{H}^*(\Delta_{g,n};\Q)\cong H^*_c(\Delta_{g,n}^{\mathrm{pure}};\Q),\] since $\Delta_{g,n}\setminus \Delta_{g,n}^{\mathrm{pure}}$ is contractible \cite[Theorem 1.1(1)]{cgp-marked}.  Both ordinary and compactly supported cohomology of $\Delta_{g,n}^{\mathrm{pure}}$ are thus interesting to study, and neither appears to be directly deducible from the other.  In particular, one does not have Poincar\'e duality for $\Delta_{g,n}^{\mathrm{pure}}$, which is almost never a manifold.
\end{rmk}

\subsection{Forgetting marked points}
\label{Mtrop}

Our approach begins with an unexpectedly useful observation in \cite{cgp-graph-homology, cgp-marked}, that certain large subspaces of $\Delta_{g,n}$ are contractible.  In particular, when $g>0$ and $n\geq 0$, consider the subspace 
$\Delta_{g,n}'$ of $\Delta_{g,n}$ parametrizing isomorphism classes of tropical curves which are obtained from a $2$-connected genus $g$ graph by adding $n$ distinct marked points.  
Then $\Delta_{g,n}'$ admits a forgetful map to $\Delta_{g}'$, and the complement $\Delta_{g,n}\setminus \Delta_{g,n}'$ is a  contractible subcomplex of $\Delta_{g,n}$ \cite[Theorem 1.1(4)]{cgp-marked}. Hence the appearance in~\eqref{eq:the-ss} of the groupoid $\Gamma^{(2)}_g$ that involves only stable $2$-connected graphs of genus $g$.  There is much earlier precedent  in the literature for the contractibility of loci of moduli spaces of graphs with cut vertices: see the work of Conant-Gerlits-Vogtmann \cite{conant-gerlits-vogtmann-cut}.
In fact one could drop the restriction to $2$-connected graphs in the above discussion, and obtain an analogous weaker theorem. However, the feasibility of using the theorem to obtain the results in Tables~\ref{table:genus3} and~\ref{table:sign_trivial} heavily relies on this further reduction to $2$-connected graphs, because it significantly reduces the number of graphs involved.

To interpret the fibers of $\Delta_{g,n}'\to \Delta_g'$ as configuration spaces on graphs, and justify~\eqref{eq:serre ss} rigorously, 
one needs to work with topological stacks rather than spaces. 
Tropical moduli stacks have indeed been studied; see \cite{conant-hatcher-kassabov-vogtmann-assembling} and the references therein for a functorial description of the moduli stack of tropical curves. 
Here, we present the stacky $\Delta_{g,n}'$ as a global quotient of the 2-connected locus of the universal family of $n$-marked genus $g$ graphs over Outer Space $\CV_{g}$ of \cite{culler-vogtmann} by the action of $\Out(F_g)$, and we use the formalism of $\Out(F_g)$-invariant pushforward, as in Lazarev-Voronov \cite{LV}.  

Note, by the way, that the existence of a forgetful map $\Delta'_{g,n}\to\Delta'_g$ is in contrast with the fact that there is no forgetful map $\Delta_{g,n}\to\Delta_g$ when $n>1$. 
For example, take a tropical curve with one edge separating two marked points on a genus $0$ vertex from the rest; when one of the two marked points is forgotten, the result is an unstable tropical curve whose stabilization is a graph with no edges, which is not a point in $\Delta_g$. This has an analogue in Harvey's curve complex of a surface $\mathcal{C}_{g,n}$, where a curve bounds a disc with two punctures: forgetting punctures results in a disallowed configuration of curves, and so there is no forgetful map $\mathcal{C}_{g,n}\to \mathcal{C}_g$. These 
curves play an important role in recent work of Brendle--Broaddus--Putman \cite{brendle-broaddus-putman}. 

There {\em is}, however, a forgetful map $\rho\col \mathcal{M}_{g,n}^\mathrm{trop}\to \mathcal{M}_{g}^\mathrm{trop}$, where $\mathcal{M}_{g,n}^{\mathrm{trop}}$ denotes the unnormalized moduli space of genus $g$, $n$-marked tropical curves, i.e., without restricting graphs to have unit volume.  The forgetful map resolves unstable vertices of weight $0$ by repeatedly collapsing all leaf edges and then ``smoothing-out" all vertices of valence $2$. 

Indeed, Hainaut and Petersen suggested an alternative construction of a tropical spectral sequence associated to this forgetful map on unnormalized tropical moduli spaces. It comes by using the Leray spectral sequence in compact support cohomology,
\begin{equation}\label{eq:tropical Leray}
E_2^{p,q}=\Hc^p(\mathcal{M}_{g}^\mathrm{trop}; R^q\rho_!\Q) \xRightarrow{\phantom{long implies}} \Hc^{p+q}(\mathcal{M}_{g,n}^\mathrm{trop};\Q)
\end{equation}
Since $\mathcal{M}_{g,n}^{\mathrm{trop}}$ is simply a generalized cone complex whose link is $\Delta_{g,n}$, these  two spaces have the same compactly supported cohomology up to a degree shift. The stalk of the sheaf $R^q\rho_!\Q$ at a tropical curve $x\in \mathcal{M}_{g}^\mathrm{trop}$ is the $q$-th compact support cohomology of the locus of all $n$-marked curves that stabilize to $x$ after forgetting marked points. This locus is a large and complicated combinatorial space, but Hainaut and Petersen outlined to us an argument comparing the latter stalk with $\Hc^*(\Conf_n(R_g);\Q)$ 
and claim the spectral sequence in \eqref{eq:tropical Leray} is isomorphic to ours.

A third spectral sequence obtained by forgetting marked points, which Hainaut--Petersen conjecture in \cite{hainaut-petersen} is isomorphic to the one above, manifests as the weight zero component of the Leray--Serre spectral sequence for the algebraic moduli spaces $\tilde{\rho}\colon \mathcal{M}_{g,n} \to \mathcal{M}_{g}$. Saito's theory of mixed Hodge modules furnishes the pushforward sheaves $R^q\tilde\rho_! \Q$ with a weight filtration, which in weight zero gives a spectral sequence
\begin{equation}\label{eq:weight zero Leray}
    E_2^{p,q} = W_0\Hc^p(\cM_g,R^q\tilde\rho_!\Q) \quad \xRightarrow{\phantom{long implies}} \quad W_0 \Hc^{p+q}(\cM_{g,n};\Q) \cong \Hc^{p+q}(\cM_{g,n}^{\mathrm{trop}};\Q)
\end{equation}
where the rightmost isomorphism is proved in \cite{cgp-marked}. Since weight zero compact support cohomology of $\cM_g$ (with constant coefficients) agrees with that of $\cM_g^{\mathrm{trop}}$, Hainaut--Petersen conjecture that the latter spectral sequence is naturally isomorphic to the one obtained from $\rho\colon \cM_{g,n}^{\mathrm{trop}}\to \cM_g^{\mathrm{trop}}$.

In this context, an advantage of Theorem~\ref{thm:main} is that passing to the subspace of $\Delta_{g,n}$ denoted by $\Delta'_{g,n}$ above, plus additional 
arguments proven in Appendix~\ref{appendix}, make it actually possible to carry out calculations, as  presented in Table~\ref{table:genus3}.  

\subsection{Relationship with hairy graph complexes and embedding spaces}\label{sec:related}

It is worth noting that Theorem~\ref{thm:main} and Tables~\ref{table:genus3} and~\ref{table:sign_trivial} equivalently compute the homology of a version of a {\em hairy graph complex}, with labelled hairs.  Indeed, our Theorem \ref{thm:main} can be understood, up to some details that we sketch below, as an equivalent reformulation of an earlier spectral sequence of Turchin--Willwacher. We thank Thomas Willwacher for pointing out this connection to us, and to Victor Turchin for additional helpful references.  Below, we attempt to elaborate on the connection, since it has not been fully articulated in the literature, and furnishes additional context and applications for our work. 

For integers $m$ and $N$, let $\mathrm{HGC}_{m,N}$ denote the (commutative) {\em hairy graph complex}, constructed by Arone--Turchin \cite{arone-turchin-graph-complexes}. It depends, up to degree shift, only on the parities of $m$ and $N$.  The complex $\mathrm{HGC}_{m,N}$ is a rational chain complex generated by connected (multi)graphs whose vertices are partitioned into a set of {\em external vertices}, of valence $1$, and {\em internal vertices}, of valence $\ge 3$.  These are oriented by choosing an orientation of the real vector space generated by the set of all external vertices (in degree $-m$), internal vertices (in degree $-N$), and edges (in degree $N-1$), together with an orientation of each edge.   If $N$ is even, the orientations of each edge play no role and can be ignored.  Two graphs, together with total ordering of their {\em orientation set} described above, have their corresponding generators identified up to sign, if there is an isomorphism of one graph to the other. The sign depends on the sign of the permutation taking one total ordering to the other, taking the degrees $-m$, $-N$, and $N-1$ mentioned above into account. The differential on $\mathrm{HGC}_{m,N}$ is a signed sum of $1$-edge contractions.  

Hairy graph complexes were introduced and used to study the rational homotopy of spaces of long embeddings \cite{arone-turchin-graph-complexes}. Let $\mathrm{Emb}_c(\R^m, \R^N)$ denote the space of embeddings $\R^m\to \R^N$ that coincide with a fixed linear embedding outside a compact subset of $\R^N$, let $\mathrm{Imm}_c(\R^m,\R^N)$ be the space of immersions with the same property, and finally let $\overline{\mathrm{Emb}}_c(\R^m, \R^N)$ denote the homotopy fiber of the map $\mathrm{Emb}_c(\R^m, \R^N)\to \mathrm{Imm}_c(\R^m,\R^N)$.  Then \cite{arone-turchin-graph-complexes} proves that the rational homotopy of $\overline{\mathrm{Emb}}_c(\R^m, \R^N)$ is isomorphic to $\Ho_*(\mathrm{HGC}_{m,n})$, for $N\ge 2m+2$.  See \cite{conant-costello-turchin-weed-two} for helpful discussion, as well as a computation of the $2$-loop part of hairy graph homology.  See~\cite{KWZ} for further results on hairy graph cohomology, including the construction of a ``waterfall'' spectral sequence that is a source  of new hairy graph cohomology classes.

In the literature, hairy graph complexes are bigraded by {\em Hodge degree} and {\em complexity}. The Hodge degree of a graph $\G$ is the number $n$ of external vertices, and the complexity is $g+n-1$, where $g$ is the first Betti number of $\G$.  Write $\mathrm{HGC}_{m,N}(s,t)$ for the complexity $s$, Hodge degree $t$ part.
For reference, and to provide context for our work, we record the following comparison isomorphisms. 
\begin{prop}\label{prop:three-comparisons} \mbox{}\begin{enumerate}
    \item For $m$ odd and $N$ even, and for all $g$ and $n$, we have canonical isomorphisms
\[\mathrm{Gr}^W_{2d}\Ho^{2d-i}(\cM_{g,n};\Q)_\mathrm{triv} \cong \widetilde{\Ho}_{i-1}(\Delta_{g,n};\Q)_\mathrm{triv} \cong \Ho_{i-2g}(G^{(g,n)})_\mathrm{triv} \cong \Ho_t(\mathrm{HGC}_{m,N}(g\!+\!n\!-\!1,n))\]
where $d = \dim \cM_{g,n} = 3g\!-\!3\!+\!n$ and $t = -mn + (i+n)(N-1) - (i-g+1)N$.
\item For $m$ and $N$ both even, we have 
canonical isomorphisms
\[\mathrm{Gr}^W_{2d}\Ho^{2d-i}(\cM_{g,n};\Q)_\mathrm{sgn} \cong \widetilde{\Ho}_{i-1}(\Delta_{g,n};\Q)_\sgn \cong \Ho_{i-2g}(G^{(g,n)})_\sgn \cong \Ho_t(\mathrm{HGC}_{m,N}(g\!+\!n\!-\!1,n))\]
for the same values of $d$ and $t$.
\end{enumerate}
\end{prop}
\noindent Here, $G^{(g,n)}$ is the {\em $n$-marked Kontsevich commutative graph complex}, with conventions as in \cite[\S2.4]{cgp-marked}; it also coincides with the complex denoted $M_g(P^n_2)$ in \cite{tsopmene-turchin-euler}. For $\lambda \in \{\mathrm{sgn}, \mathrm{triv}\}$, the subscript $\lambda$ denotes the $\lambda$-isotypical component of this $S_n$-representation.  

In each item of Proposition~\ref{prop:three-comparisons}, the first isomorphism holds without restriction to $\lambda = \mathrm{sgn},\mathrm{triv}$ and is known from \cite{acp, cgp-marked}, via Deligne's mixed Hodge theory. The second also holds without restriction, and is in \cite{cgp-marked}. The third is easily deduced as follows; the hardest part is keeping track of degrees. By turning markings on vertices of a marked graph into new external vertices in a hairy graph, and forgetting the labels, we convert an $n$-marked graph of genus $g$ having $i$ edges to a hairy graph with $n$ external vertices, $n+i$ edges, and $v = i-g+1$ internal vertices.  The condition that $N$ is even corresponds to the orientation conventions for $G^{(g,n)}$.  If in addition $m$ is even, then each of the $n$ markings  yields odd total degree $-m+(N-1)$ in the hairy graph complex, corresponding to the degree of the univalent vertex plus the degree of the incident edge that together replace the marking.  If $m$ is odd, then $-m+(N-1)$ is even.  

In this way, our results, Table~\ref{table:sign_trivial} in particular, give new data about embedding spaces into $\R^N$ when $N$ is even. And, what we are studying in Theorem~\ref{thm:main} and Table~\ref{table:genus3} can be regarded as a {\em labelled} version of the hairy graph complex. Such a variant was discussed in~\cite[Remark 6.2]{turchin-willwacher-commutative} but not carried out there any further.  Of course, the bridge from labelled hairy graph complexes to tropical moduli spaces and top-weight cohomology of $\cM_{g,n}$, central to our perspective, was also not in place in the literature at the time that~\cite{turchin-willwacher-commutative} was written.  This bridge was provided by Chan--Galatius--Payne's contractibility of the ``positive vertex weight'' subspace of $\Delta_{g,n}$, denoted $\Delta_{g,n}^{\mathrm{w}}$ \cite[Theorem 1.1(1)]{cgp-marked}.

Finally, the explicit connection of our Theorem~\ref{thm:main} to \cite{turchin-willwacher-commutative} is as follows.  
They describe a spectral sequence converging to hairy graph homology, whose $E_1$-page is a (nonhairy) graph complex with coefficients in the local system of Hochschild--Pirashvili cohomology of graphs.
In turn, rational Hochschild--Pirashvili cohomology of a finite simplicial complex $X$ is known to be essentially equivalent to the compactly supported cohomology $\Ho^*_c(\Conf_n(X);\Q)$ for all $n\geq 1$ (see, e.g. \cite[Theorem 1.12]{gadish-hainaut}). Putting the two equivalences together, our Theorem \ref{thm:main} can a posteriori be understood as an adaptation of a labelled version of Turchin--Willwacher's spectral sequence to the tropical moduli space $\Delta_{g,n}$.

\subsection{What's new}\label{sec:whats-new} Having compared, we now contrast with \cite{turchin-willwacher-commutative}, and highlight additional new aspects of our work.  
First, in this paper, we leverage the understanding of $\Ho^*_c(\Conf_n(R_g);\Q)$ from \cite{bibby-chan-gadish-yun-homology, gadish-hainaut} to obtain new calculations, for low genus $g\le 3$, but for higher $n$ than previously attainable.  Also, the reduction to $2$-connected graphs is also new and crucial for obtaining these calculations (Section~\ref{sec:outer-space-and-friends}).  Additional important computational aids are discussed in Appendix~\ref{appendix}.

In addition, we offer some new theorems on the level of geometry.  We first establish in Section~\ref{sec:cech} a general formula for sheaf cohomology of an ideal simplicial complex equivariant with respect to the action of a discrete group (Proposition~\ref{prop:formula for compact support cohomology}). Then Corollary~\ref{cor:geometric cech bicomplex} provides a geometric application, to instances when such a sheaf arises from a space-level map to an ideal simplicial complex that is ``simplicially locally trivial.''  
We expect that the above results will apply readily to other moduli spaces in tropical geometry.  

In fact, roughly speaking, when there are results available about contractibility, or at least homological acyclicity, of ``repeated marking subcomplexes'' (e.g., \cite{cgp-marked}), it is reasonable to hope for a  spectral sequence similar to~\eqref{eq:the-ss} that also uses local systems coming from cohomology with compact supports of configuration spaces. For example, we expect applications to moduli spaces of Hassett-weighted tropical curves $\Delta_{g,w}$ \cite{cavalieri-hampe-markwig-ranganathan, kannan-li-serpente-yun-topology, ulirsch-tropical}, which should be related in this way to $w$-weighted configuration spaces of graphs.  Second, we also expect applications of the results in Section~\ref{sec:cech} to moduli spaces of admissible $G$-covers of curves, related to configuration spaces of points on $G$-covers of graphs, in light of the recent work \cite{brandt-chan-kannan-on}.  That paper constructs boundary complexes of moduli spaces of admissible $G$-covers of genus $0$ curves, when $G$ is abelian, and, important to this discussion, proves that an appropriate repeated marking locus is acyclic in homology.

\subsection*{Acknowledgements}
We are grateful for support we received via the program Collaborate@ICERM, which provided ideal working conditions at the Institute for Computational and Experimental Research in Mathematics (ICERM) during the summer of 2022.

Furthermore, this work benefited from many useful conversations with D.~Petersen and L.~Hainaut, whose suggestions inspired the proof of our main theorem.

V.~Turchin and T.~Willwacher provided important feedback on an earlier draft, including crucial references; we thank them very much.

C.B. was supported by NSF DMS-2204299; M.C. was supported by NSF CAREER DMS-1844768, a Sloan Foundation Fellowship and a Simons Foundation Fellowship; N.G. was supported by NSF DMS-1902762 and an AMS--Simons Travel Grant.

\section{Graph complexes}\label{sec:graph complexes}

Decorated graph complexes were introduced in \cite[\S4]{turchin-willwacher-commutative}, constructing a certain chain complex from a representation of the outer automorphism group $\Out(F_g)$. Below, we need a slight variant of their construction, so we take this section to recall the notion and set up the conventions we use.

As we shall explain, our graph complexes arise from local systems on the moduli {\em stack} of metric graphs of genus $g$. This stack is realized as a stack quotient $[\CV_g/\Out(F_g)]$, where $\CV_g$ denotes Culler-Vogtmann Outer Space of rank $g$. A close comparison in the literature is
Lazarev--Voronov's \cite{LV} 
{\em constructible sheaves} on the moduli \emph{space} of metric graphs $\CV_g/\Out(F_g)$, rather than the stack -- see Remark \ref{rmk:LV definitions} below.

\subsection{Graph complexes associated to graph local systems.}

We work over $\Q$, or equally well, over any field of characteristic $0$.  We will freely regard ungraded vector spaces as graded vector spaces concentrated in degree $0$.  Given a vector space $W$ of dimension $n$, write $\det W = \Lambda^n(W)[-n]$ for the determinant: it is $1$-dimensional, concentrated in degree $n$.  If $E$ is a finite set, for example the set of edges of a graph, set $\det E = \det \Q^E$. Then $\det$ is a contravariant functor from finite sets to vector spaces.  

\begin{defn}Let $\GrphCat_g$ be the category of connected (multi-)graphs of genus $g$, with morphisms generated by isomorphisms and non-loop edge contractions $G\to G/e$. 
\end{defn}

\begin{defn} Let $\A$ be any (small) full subcategory of $\GrphCat_g$ that is closed under isomorphisms; call its objects {\em $\A$-graphs} (e.g., 2-connected graphs, or loopless graphs). Say that $\cA$ has the {\em diamond property} if 
for all graphs $\G$ and $e,f\in E(\G)$, if $\G, \G/e,$ and $\G/\{e,f\}$ are in $\A$, then $\G/f\in \A$.
\end{defn}
\noindent Call a full subcategory $\cA$ of $\GrphCat_g$ {\em interval-closed} if whenever $\G, \G'\in \cA$ and $\G\to \G''\to \G'$ is a diagram in $\GrphCat_g$, then $\G''\in \A$. 
If $\cA$ is interval-closed, then $\cA$ has the diamond property. 

\begin{defn}\label{def:local-system}
A functor \[\V\col \GrphCat_g^\op \to \Vect^{\cong}_\Q\] into vector spaces and linear isomorphisms is called a \emph{graph local system}.  

For $\cA$ a full subcategory of $\GrphCat_g$, a functor \[\V\col \A^\op \to \Vect^{\cong}_\Q\] is called an \emph{$\A$-graph local system}.  
For an $\cA$-graph $G$, we shall write $V_\G$ rather than $\V(\G)$.
\end{defn}

Let $\cA$ have the diamond property.  We define an exact functor $\GC(\A;-)$ that takes an $\A$-graph local system $\V$ to a chain complex $\GC(\A;\V)$, called the  $\cA$-{\em graph complex} of $\V$.  Thus a short exact sequence of graph local systems yields a long exact sequence of cohomology of graph complexes.
Let $\A^{\cong}$ denote the groupoid which has the same objects as $\A$ and whose morphisms are the isomorphisms in $\A$.
\begin{defn}\label{def:local-system-complex}
Suppose $\A$ has the diamond property and let $\V\col \A^\op \to \Vect_{\Q}^{\cong}$ be an $\A$-graph local system.  The \emph{$\mathcal{A}$-graph complex} of $\V$ is the chain complex
\begin{equation}\label{eq:gc-a-v}
    \GC(\A; \V) = \varprojlim_{\G\in \A^{\cong}} \left(V_\G \otimes \det(E(\G))[1]\right)
\end{equation}
with limit taken over the linear maps
\begin{equation}\label{eq:intertwiner}
    \V_\phi \otimes \det \phi\col V_{\G_2} \otimes \det E(\G_2) \xrightarrow{\ \cong\ } V_{\G_1} \otimes \det E(\G_1)
\end{equation}
induced by graph isomorphisms $\phi\col G_1\to G_2$.
The differential, of degree $+1$, is given by edge uncontraction, described precisely in the next definition.
\end{defn}

\begin{rmk}\label{rmk:direct-sum-way}
By picking a representative from every isomorphism class of $\A$-graphs, 
the complex can be explicitly presented in each degree $p$ as 
\begin{equation}\label{eq:graph-complex}
    \GC^p(\A; \V) \cong \bigoplus_{\substack{[\G] \in \operatorname{Iso}(\A)\\ |E(\G)| = p+1}} \left( V_\G \otimes \det(E(\G)) \right)^{\Aut(\G)}.
\end{equation}
Note the direct sum replacing a direct product since there are finitely many connected graphs with $p$ edges, up to isomorphism.
\end{rmk}

We now describe the differential in an $\A$-graph complex. Given an $\A$-graph $\G$, say that $e\in E(\G)$ is an \emph{$\A$-edge} if $\G/e$ is an $\A$-graph.  Write $E_\A(\G)$ for the set of $\A$-edges.
For every $\A$-graph $\G$ and $e\in E_{\A}(\G)$, let
\begin{equation}\label{eq:edge component of graph diff}
    d_e: V_{\G/e} \otimes \det(E(\G/e)) \longrightarrow V_\G \otimes \det(E(\G))
\end{equation}
denote the tensor product of the linear map $\V_{\G\to \G/e}$ with $e\wedge(-):\det(E(\G/e)) \xrightarrow{\cong} \det(E(\G))$. Then $d_e$ is a linear isomorphism.

For any $\A$-graph $\G'$, write $\pi_{\G'}: \GC(\cA;\V) \to V_{\G'} \otimes \det(E(\G'))$  for the canonical map out of the inverse limit.
Summing over $e\in E_{\A}(\G)$ yields a map
\begin{equation}\label{eq:sum-differential}
\sum_{e\in E_\A(\G)} d_e\circ \pi_{\G/e} \colon \GC(\A;\V) \to V_{\G} \otimes \det (E(\G))\end{equation}
for each $\A$-graph $\G$. A quick verification shows that the maps~\eqref{eq:sum-differential} commute with the isomorphisms $V_G\otimes \det E(\G) \to V_{G_2}\otimes \det E(\G_2)$ induced by graph isomorphisms $\G\to \G_2$.  

\begin{defn}
The differential on $\GC(\cA;\V)$  in~\eqref{eq:gc-a-v} is defined as the map 
\begin{equation}\label{eq:d}
d\col \GC(\A;\V)\to \GC(\A;\V)
\end{equation}
that is induced by from \eqref{eq:sum-differential} by the universal property of limits.
\end{defn}

\begin{lem}\label{lem:d2}
With $d$ as in~\eqref{eq:d}, we have $d^2  = 0.$
\end{lem}

\begin{proof}
The usual proof that a chain complex differential squares to 0 applies when $\cA$ has the diamond property:
one may contract any pair of $\cA$-edges in either order, with the two possible orders contributing opposite sign. 
\end{proof}

Later, in Remark~\ref{rem:a-paper-differential} we provide a formula for the differential in terms of the choices of representatives in Remark~\ref{rmk:direct-sum-way}.

\begin{rmk}\label{rmk:LV definitions}
We compare the setup in this section with the related work of Lazarev--Voronov \cite{LV}, who construct a graph complex associated to a cyclic operad. Their construction starts with a functor from the category of graphs to vector spaces: 
a \emph{coefficient system} associated to a constructible sheaf on the moduli space of metric graphs. 
In particular, they allow non-invertible morphisms between vector spaces associated with different graphs, which gives flexibility: e.g., one obtains all graph complexes arising from cyclic operads \cite{conant-vogtmann-on}.

The local systems considered here only involve invertible linear maps rather than arbitrary linear maps, but we work with the {\em stack} of metric graphs, 
allowing the group $\Aut(\G)$ to act nontrivially on the vector space $V_\G$ associated to a graph $\G$. Hence invariant subspaces appear in \eqref{eq:graph-complex}.  
A graph local system yields a coefficient system in the sense of \cite{LV} in the special case that every automorphism $\phi$ of a graph $\G$ induces the identity map on $V_\G$. 
\end{rmk}

\subsection{Equivalence of graph local systems and $\Out(F_g)$-representations}
Now consider the case that $g>0$ and $\Gamma_g$ is the subcategory of graphs with each vertex of valence at least 3.
Right representations of the group $\Out(F_g)$ furnish immediate examples of $\Gamma_g$-graph local systems, as noted by Turchin--Willwacher:

\begin{prop}[{\cite[Example 4]{turchin-willwacher-commutative}}] \label{prop:equiv-cat} The category of $\Gamma_g$-graph local systems is equivalent to the category of right $\Out(F_g)$-representations.
\end{prop}

Suppose $\A$ is a full subcategory of $\Gamma_g$ and has the diamond property.  Let ${V}$ be a right $\Out(F_g)$-representation.  In view of Proposition~\ref{prop:equiv-cat}, ${V}$  gives rise to a graph complex $\GC(\A;{V})$, spelled out below.
\begin{defn}
  The \emph{$\mathcal{A}$-graph complex with coefficients in} ${V}$ is the chain complex $\GC(\A;\V)$ of Definition~\ref{def:local-system-complex} associated to the graph local system $\V$ of ${V}$ as in Proposition~\ref{prop:equiv-cat}. Precisely,
\begin{equation}
    \GC(\A; {V}) = \varprojlim_{\G\in \A^{\cong}} \left({V} \otimes \det(E(\G))[1]\right),
\end{equation}
with inverse limit defined as follows.  For each $\A$-graph $\G$, pick a marking $m_G\col R_g\xrightarrow{\sim}  \G$ and a homotopy inverse $m_G^{-1}$. Then for every isomorphism $\phi\col G_1\to G_2$, the composition $m_{\phi} = m_{G_2}^{-1}\circ \phi \circ m_{G_1}:R_g\to R_g$ represents an element in $ \operatorname{hEq}(R_g,R_g)\cong \Out(F_g)$, depending functorially on $\phi$. Tensoring with $\det_{\phi}$ furnishes an isomorphism
\[(m_{\phi})^* \otimes {\det}_{\phi} : {V} \otimes \det(E(\G_2)) \to {V} \otimes \det(E(\G_1))\]
and the inverse limit is taken over these maps.
\end{defn}

\begin{rmk}\label{rmk:direct-sum-way-out}
As in Remark~\ref{rmk:direct-sum-way}, picking a representative from every isomorphism class of $\A$-graphs, the complex can be explicitly presented in each degree $p\ge -1$ as 
\begin{equation}
    \GC^p(\A; {V}) \cong \bigoplus_{\substack{[\G] \in \operatorname{Iso}(\A) \\ |E(\G)|=p+1}} \left( {V} \otimes \det(E(\G)) \right)^{\Aut(\G)}
\end{equation}
where $\Aut(\G)$ acts as in the previous definition: through a marking $m_G: R_g \to G$.
\end{rmk}

\begin{rmk}[Cohomology of moduli stack of graphs]
As explained in \cite[\S4.1]{turchin-willwacher-commutative}, a $\Gamma_g$-graph local system $\V$ is equivalent to a local system on the moduli stack of unit-volume connected metric graphs of genus $g$. This stack may be presented as the global quotient stack of the Culler--Vogtman Outer Space $\left[\CV_g/\Out(F_g)\right]$ (recalled in \S\ref{sec:outer-space-and-friends}), and a subcategory of $\A$-graphs defines a substack. 

    The graph complex given above computes the cohomology of the stack with compact support and coefficients in the local system $\V$. Generally, the cohomology of this stack is given by the homotopy limit $\operatorname{holim}_{G\in \Gamma_g} V_\G$, and the graph complex is a cellular chain complex computing it.
    Note that one may take ordinary $\Aut(\G)$-invariants in \eqref{eq:graph-complex} instead of derived invariants since $V_\G$ consists of rational vector spaces and $\Aut(\G)$ is a finite group, so every representation is injective.
\end{rmk}

More generally, a graph local system $\V$ can be taken to be a complex of graph local systems, sending every morphism in $\A$ to a quasi-isomorphism of complexes. In this case the $\A$-graph complex is naturally a double complex, whose total complex computes the graph cohomology with coefficients in $\V$. This is the level of generality that we use in \S\ref{sec:proof} for the proof of Theorem \ref{thm:ss}.

\section{Equivariant constructible sheaves on ideal simplicial complexes}  
\label{sec:cech}

Towards a proof of Theorem~\ref{thm:main}, we establish a formula for sheaf cohomology of an ideal simplicial complex equivariant with respect to the action of a discrete group.  Throughout, our sheaves are sheaves of $\Q$-vector spaces or complexes of sheaves of $\Q$-vector spaces.

An \emph{ideal simplicial complex} is the complement $U$, in a simplicial complex $X$, of a closed subcomplex $Z$.  It can thus be presented as topological space as a colimit of a diagram of ideal simplices and face morphisms.  Here, an \emph{ideal simplex} is a simplex possibly with some of its proper closed faces removed. 
Note that if $U$ is an ideal simplicial complex then it is the union $U$ of the open stars of simplices of $X$ not in $Z$.  Conversely, every union $U$ of open stars of simplices in $X$ is an ideal simplicial complex.

Now let $\Gp$ be a discrete group.  
Let $U$ be a space with a $\Gp$-action, and let $f\colon U \to U/\Gp$ be the quotient map.  Let $\F$ be a $\Gp$-equivariant sheaf on $U$.  

\begin{defn}
The {\em invariant pushforward} $f_*^\Gp\F$ of $\F$ is the sheaf on $U/\Gp$ given by
\[
f_*^\Gp\F(V) := \F(f^{-1}(V))^\Gp = (f_*\cF)(V)^\Gp, 
\]
for each open  $V\subseteq U/\Gp$.
\end{defn}
\noindent Note that if the action of $\Gp$ on $U$ is proper, the stalks $(f_*^\Gp\F)_{[x]}$ of the invariant pushforward are described up to natural isomorphism as \[(f_*^\Gp\F)_{[x]} \cong (\F_x)^{\stab_\Gp(x)},\]
for $x\in U$.

Let $X$ be a space with a $\Gp$-action and $U\subseteq X$ a $\Gp$-invariant open subset on which the action of $\Gp$ is proper. Consider the commutative square
\[
\xymatrix{ U \ar[r]^{\bar{f}} \ar@{^{(}->}[d]_{\bar{j}} & U/\Gp \ar@{^{(}->}[d]^j \\ X \ar[r]^{{f}} & X/\Gp. }
\]

\begin{lem}\label{lem:the-previous-lemma}
There is a natural isomorphism of functors 
\begin{equation}\label{eq:the-functors} j_! \bar{f}_*^\Gp \cong {f}_*^\Gp\bar{j}_!\end{equation}
from the category of $\Gp$-equivariant sheaves on $U$ to the category of sheaves on $X/\Gp$.  Moreover, the  
functors~\eqref{eq:the-functors} are exact. 
\end{lem}
\begin{proof}
    Fix a $\Gp$-equivariant sheaf $\F$ on $U$, and recall that $\bar{j}_!\F$ is the $\Gp$-equivariant sheaf on $X$  obtained as the sheafification of the presheaf $\bar{j}_!^p $ given by
    \[
    \bar{j}_!^p\F(V) = \begin{cases} \F(V) & \text{ if } V\subseteq U \\
    0 & \text{ otherwise},
    \end{cases}
    \]
    and similarly for $j_!$. By the universal property of sheafification, to give a map $j_!\bar{f}_*^\Gp \F \to {f}_*^\Gp\bar{j}_! \F$ is equivalent to giving a morphism of presheaves $j_!^p \bar{f}_*^\Gp \F \to {f}_*^\Gp\bar{j}_! \F$, which we produce next.

    For an open $V \subseteq X/\Gp$ we have $j_!^p \bar{f}_*^\Gp \F(V) = 0$ unless $V\subseteq U/\Gp$, in which case
    \[
    j_!^p (\bar{f}_*^\Gp \F)(V) = \bar{f}_*^\Gp \F(V) = (\F(\bar{f}^{-1}(V)))^\Gp 
    = (\F({f}^{-1}(V)))^\Gp
    \cong (\bar{j}_!\F({f}^{-1}(V)))^\Gp = {f}_*^\Gp\bar{j}_!\F (V).
    \]
    In either case, there is a canonical morphism $j_!^p \bar{f}_*^\Gp \F(V) \to {f}_*^\Gp\bar{j}_! \F(V)$, and these are clearly compatible with restrictions $V'\subset V$. Thus there is an induced map of sheaves
    \begin{equation*}
        j_!^p \bar{f}_*^\Gp \F \to {f}_*^\Gp\bar{j}_! \F.
    \end{equation*}It is an isomorphism since on stalks,
    \[
    (j_!\bar{f}_*^\Gp \F)_{[x]} \cong \begin{cases} (\F_x)^{\stab_\Gp(x)} & \text{ if } [x]\in U/\Gp \text{ with orbit representative } x\in U\\
    0 & \text{ if } [x]\notin U/\Gp 
    \end{cases}
    \cong ({f}_*^\Gp\bar{j}_! \F)_{[x]}
    \]
    with isomorphism induced by the morphism we defined. The two sheaves are thus isomorphic.

    Exactness of the functors~\eqref{eq:the-functors} follows from the facts that $j_!$ is always exact, and $\bar{f}_*^\Gp$ is exact for quotients by a proper action, as can be seen by the above description of stalks: properness implies that the groups $\stab_\Gp(x)$ are finite for all $x\in U$, and taking invariants of finite groups on $\Q$-vector spaces is exact. 
\end{proof}

Let $X$ be a simplicial complex.  A {\em constructible sheaf} $\F$ on $X$ is a sheaf that is constant on every open simplex: see \cite{kashiwara-schapira}, \cite{LV}.  Let $P(X)$ denote the poset of simplices in $X$, with a unique arrow $\tau\to \sigma$ whenever $\tau$ is a face of $\sigma$.  Recall that the category of constructible sheaves on $X$ is equivalent to the category of covariant functors from $P(X)$ to chain complexes. The equivalence is given by associating, to a constructible sheaf $\F$ on $X$, the functor
\[\tau \mapsto \F_\tau := \F(\st(\tau))\] on $P(X)$, where $\st(\tau) = \bigcup_{\tau\subseteq \sigma}\sigma^\circ$ is the open star of $\tau\in P(X)$.

If $U\subseteq X$ is a union of open stars of simplices, denote $P(U)$ for the subposet of simplices in $U$. By a \emph{constructible sheaf} on $U$ we always mean one that is constant on open simplices. It is therefore equivalent to a functor $\tau \mapsto \F_\tau$ defined on $P(U)$.
Now, suppose that a discrete group $\Gp$ acts on $U$ properly, and $\F$ is a $\Gp$-equivariant constructible sheaf on $U$.  We give a formula for the compactly supported cohomology of $U/\Gp$ with coefficients in $f_*^\Gp(\F)$ in the next proposition.  It is a natural generalization of an analogous result in Lazarev--Voronov's \cite[Proposition 2.4]{LV}, which considers only the case of $\Gp$-equivariant sheaves on $U$ that are pulled back from $U/\Gp$. 

Recall for a finite set $\tau$, we denote by $\det(\tau)$ the exterior power $\wedge^{|\tau|} \Q^\tau$, placed in degree $|\tau|$.  This is naturally a representation of the symmetric group permuting the elements of $\tau$. When $\tau$ is a simplex of a $\Gp$-equivariant simplicial complex $X$, then $\det(\tau)$ is a representation of $\stab_\Gp(\tau)$.
\begin{prop}[A \v{C}ech complex for proper quotients]\label{prop:formula for compact support cohomology}
    Let $X$ be a $\Gp$-simplicial complex such that $X/\Gp$ is compact, and let 
    $U\subseteq X$
    be a $\Gp$-invariant ideal subcomplex of $X$ on which the $\Gp$-action is proper.
    Let $f\col U\to U/\Gp$ 
    be the quotient map.
    
    For $\F$ a complex of $\Gp$-equivariant constructible sheaves on $U$, 
    corresponding to $\tau\mapsto \F_\tau$, 
    $$
    R\Gamma_c(U/\Gp; f_*^\Gp\F) \simeq \bigoplus_{[\tau]\in P(U)/\Gp} \left(\F_{\tau} \otimes \det(\tau)[1]\right)^{\stab_\Gp(\tau)},
    $$
    where complex on the right is the total complex obtained from a double complex
    \begin{equation}\label{eq:double}
    C^{p,q} = \bigoplus_{\substack{[\tau]\in P(U)/\Gp\\ \dim(\tau) = p}} \left( \F_\tau^q\otimes \det(\tau)\right)^{\stab_\Gp(\tau)}.
    \end{equation}
    The horizontal differential of~\eqref{eq:double} is the sum of maps $\F_{\tau} \to \F_{\{v\}\cup\tau}$ given by the functor structure of $\F$ 
    and $\tau_0\wedge \ldots \wedge \tau_p \mapsto v \wedge \tau_0\wedge \ldots \wedge \tau_p$ on the determinantal representation. The vertical differential is the internal differential of $\F$.
\end{prop}
\begin{proof}
    Using Lemma~\ref{lem:the-previous-lemma}, there is a natural equivalence of derived functors
    \[
    R\Gamma_c(U/\Gp; f_*^\Gp\F) \simeq R\Gamma(X/\Gp ; j_! f_*^\Gp \F) \simeq R\Gamma(X/\Gp ; f_*^\Gp j_! \F).
    \]
    Let $\inv_\Gp$ be the functor that takes $\Gp$-invariants of a rational $\Gp$-representation.
    Since there is a natural identification of composites $\Gamma\circ f_*^\Gp = \inv_\Gp\circ \Gamma$ and by exactness of $j_!$ and $f_*^\Gp j_!$ (Lemma~\ref{lem:the-previous-lemma}), there is an equivalence of derived functors
    \[
    R\Gamma(X/\Gp ; f_*^\Gp j_! \F) \simeq R(\inv_\Gp)\circ R\Gamma(X; j_!\F).
    \]
    We next compute $R\Gamma(X,j_!\F)$ using a \v{C}ech bicomplex. 
    
    When $\F$ is
    a constructible sheaf on $U$, 
    its extension by zero $j_!\F$ to $X$ is the constructible sheaf associated with the coefficient system
    \[
    \tau \mapsto \begin{cases}
        \F_\tau & \text{ if } \tau^\circ \subseteq U \\ 
        0 & \text{ otherwise}.
    \end{cases}
    \]
    Consider the good open cover of the simplicial complex $X$ by open stars of vertices, so that the intersection $\st(v_0)\cap \ldots\cap \st(v_p) = \st([v_0,\ldots,v_p])$ when $\{v_0,\ldots,v_p\}$ is a simplex, and is empty otherwise. The corresponding \v{C}ech bicomplex computing $R\Gamma(X,j_!\F)$ takes the form
    \[
    \prod_{\tau\in \P(X)} (j_!\F)_\tau \otimes \det(\tau)
 \cong \prod_{\tau\in \P(U)} \F_\tau \otimes \det(\tau)
    \]
    with bigrading and differentials as stated in the lemma.    

    Now, the $\Gp$-action on $\P(U)$ permutes the terms of the bicomplex, and thus identifies it with a product of (co)induced representations
    \[
    \prod_{\tau\in \P(U)} \F_{\tau}\otimes \det(\tau) \cong \prod_{[\tau]\in \P(U)/\Gp} \operatorname{Ind}_{\stab_\Gp(\tau)}^\Gp (\F_\tau\otimes \det(\tau))
    \]
    Since all the stabilizers are finite groups and $\F_\tau$ is a complex of rational vector spaces, these coinduced representations are coinductions of injectives and hence themselves injectives in the category of rational $\Gp$-representations. It follows that their derived $\Gp$-invariants coincide with the ordinary invariants
    \[
    \bigg(    \prod_{[\tau]\in \P(U)/\Gp} \operatorname{Ind}_{\stab_\Gp(\tau)}^\Gp (\F_\tau\otimes \det(\tau)) \bigg)^\Gp \cong \prod_{[\tau]\in \P(U)/\Gp} \left(\F_\tau\otimes \det(\tau)\right)^{\stab_\Gp(\tau)},
    \]
    as claimed, noting that this last product is now finite and hence isomorphic to the direct sum.
\end{proof}
A geometric application of this sheaf cohomology statement is the following.

\begin{cor} \label{cor:geometric cech bicomplex}
    Let $\Gp$ and $U\subseteq X$ satisfy the hypotheses of Proposition~\ref{prop:formula for compact support cohomology}, and let $\tilde{U}$ be any $\Gp$-space equipped with a $\Gp$-equivariant proper map $\ol{\pi}\colon \tilde{U}\to U$. 
    Suppose further that $\ol{\pi}$ is ``simplicially locally trivial" in the sense that for every open simplex $\tau^\circ \subseteq U$ there exists a topological space $F_\tau$ so that there is a pullback square
    \[
    \xymatrix{
    \tau^\circ \times F_\tau \ar[r]\ar[d] & \tilde{U} \ar[d] \\
    \tau^\circ \ar@{^(->}[r] & U.
    }
    \]
    Then $\F:=R\ol{\pi}_*\ul{\Q}_{\tilde{U}}$ is a complex of constructible sheaves on $U$, corresponding to the functor 
    \begin{equation}\label{eq:Rpi*Q constructible}
    \F_\bullet\colon \tau\mapsto R\Gamma(F_\tau;\ul{\Q}),
    \end{equation}
    so that we have a \v{C}ech complex
\begin{equation}\label{eq:corollary_cech}
    R\Gamma_c(\tilde{U}/\Gp ;\ul{\Q} ) 
    \simeq R\Gamma_c(U/\Gp;f_*^{\Gp}\F)
    \simeq \bigoplus_{\tau \in P(U)/\Gp} \left( R\Gamma(F_\tau;\ul{\Q})\otimes \det(\tau)[1] \right)^{\stab_\Gp(\tau)}
\end{equation}
with differentials as in Proposition \ref{prop:formula for compact support cohomology}.
\end{cor}
\begin{proof}
    First, $\Gp$-equivariance of $\ol{\pi}$ forces the $\Gp$-action on $\tilde{U}$ to be proper as well. Consider the commutative square
    \[
    \xymatrix{
    \tilde{U} \ar[r]^-{\ol{f}} \ar[d]_{\ol{\pi}} & \tilde{U}/\Gp \ar[d]^{\pi} \\
    U \ar[r]^-f & U/\Gp
    }
    \]
    where the horizontal arrows are the quotient maps. Properness of $\ol{\pi}$ implies that $\pi$ is proper as well.
    
    By $\ol{f}_*^{\Gp}\ul{\Q}_{\tilde{U}} \cong \ul{\Q}_{\tilde{U}/\Gp}$, there is a natural quasi-isomorphism 
    \begin{equation}\label{eq:qiso1}
    R\Gamma_c(\tilde{U}/\Gp;\ul{\Q}_{\tilde{U}/\Gp}) \simeq R\Gamma_c(\tilde{U}/\Gp;\ol{f}_*^{\Gp}\ul{\Q}_{\tilde{U}}).
    \end{equation}
    Furthermore, properness of $\ol{\pi}$ gives $\pi_! = \pi_*$ so that
    \begin{equation}\label{eq:qiso2}
    R\Gamma_c(\tilde{U}/\Gp; \ol{f}_*^{\Gp}\ul{\Q}_{\tilde{U}}) \simeq R\Gamma_c(U/\Gp; R\pi_*\ol{f}_*^{\Gp}\ul{\Q}_{\tilde{U}}).
    \end{equation}

Now, the composite functors $f_*^{\Gp} \ol{\pi}_*$ and $\pi_* \ol{f}_*^{\Gp}$ coincide, since for any equivariant sheaf $\F$ and all open $V\subseteq \Delta_g^{(2)}$ one has 
    \[
    f_*^\Gp\ol{\pi}_*\F(V) = ( f_*\ol{\pi}_*\F(V) )^\Gp = 
    ( \pi_* \overline{f}_*\F(V) )^\Gp = ( \overline{f}_*\F(\pi^{-1}(V)) )^\Gp = \pi_* \overline{f}_*^\Gp\F(V).
    \]
Since $f_*^{\Gp}$ is exact, the same commutation holds for the derived functors,
    \begin{equation}\label{eq:qiso3}
    f_*^{\Gp} \circ R\ol{\pi}_*
    \simeq 
    R\left( f_*^{\Gp} \circ \ol{\pi}_* \right) 
    \simeq 
    R\left( \pi_* \circ \ol{f}_*^{\Gp} \right) 
    \simeq 
    R\pi_*\circ \ol{f}_*^{\Gp}.
    \end{equation}
    A routine application of proper base change (see, e.g., \cite[Remark 2.5.3]{kashiwara-schapira}) shows that 
    $\F:= R\ol{\pi}_*\ul{\Q}_{\tilde{U}}$ is constructible on $U$ and corresponds to the covariant functor
    \begin{equation}\label{eq:ftau}
    \F_\bullet\colon \tau \mapsto R\Gamma(F_\tau ; \ul{\Q}).
    \end{equation}
    Composing the equivalences \eqref{eq:qiso1}, \eqref{eq:qiso2}, and \eqref{eq:qiso3}, we get
    \[    R\Gamma_c(\tilde{U}/\Gp;\ul{\Q}_{\tilde{U}/\Gp}) \simeq R\Gamma_c(U/\Gp; f_*^{\Gp}\F),
    \]
    where $\F = R\ol{\pi}_*\ul{\Q}_{\tilde{U}}$,
    so Proposition \ref{prop:formula for compact support cohomology} applies and 
    with \eqref{eq:ftau}
    yields the claimed equivalence to the \v{C}ech double-complex.
\end{proof}

\begin{rmk}\label{rmk:Cech differential using retractions}
    The sheaf structure of $\F=R\ol{\pi}_*\ul\Q_{\tilde{U}}$ from the last corollary furnishes restriction maps
    \[
    R\Gamma(F_\tau;\ul\Q) \to R\Gamma(F_\sigma;\ul\Q)
    \]
    for every face inclusion $\tau\subseteq \sigma$, determining the horizontal \v{C}ech differential in \eqref{eq:corollary_cech} as an alternating sum over these.
    Computing the restriction maps in full generality is subtle, but the following geometric input makes it possible: if for a simplex $\tau$ there exist compatible retractions
    \[
    \xymatrix{
    \ol{\pi}^{-1}(\st(\tau)) \ar[r]^-{\tilde{r}} \ar[d] & \ol{\pi}^{-1}(\tau^\circ) \ar[d] \\
    \st(\tau) \ar[r]^-{r} & \tau^\circ
    }
    \]
    then for every face inclusion $\tau\subseteq \sigma$ the restriction map $\F_{\tau} \to \F_{\sigma}$ is the induced pullback $\tilde{r}^*\colon R\Gamma(F_\tau;\ul{\Q}) \to R\Gamma(F_\sigma;\ul{\Q})$ on the fiber over any point $x\in \sigma^\circ \subseteq \st(\tau)$.

    Indeed, the inclusions $\{r(x)\}\into \tau^\circ \into \st(\tau)$ have retractions $ \{r(x)\}\overset{c}\leftarrow \tau^\circ \overset{r}\leftarrow \st(\tau)$, giving rise to the following diagram of restriction maps
    \[
    \xymatrix{
    \F_\tau = R\Gamma(\st(\tau);\F) \ar[r]_-{i_{\tau^\circ}^*} \ar@{<--}@/^1pc/[r]^-{r^*} & 
    R\Gamma(\tau^\circ;\F|_{\tau^\circ}) \ar[r]_-{i^*_{r(x)}} \ar@{<--}@/^1pc/[r]^-{c^*} & 
    \F_{r(x)}
    }
    \]
    with $r^*$ and $c^*$ are left-inverses to their respective restriction map $i_{\tau^\circ}^*$ and $i^*_{r(x)}$. But $\F$ is constructible, so $i_{\tau^\circ}^*$ and $i^*_{r(x)}$ are quasi-isomorphisms, and thus $r^*$ and $c^*$ are also homotopy right-inverses. 
    
    On the other hand, 
    this diagram is naturally quasi-isomorphic to
    \[
    \xymatrix{
    R\Gamma(\ol{\pi}^{-1}(\st(\tau));\ul\Q) \ar[r]_-{i_{\ol{\pi}^{-1}(\tau^\circ)}^*} \ar@{<--}@/^1pc/[r]^-{\tilde{r}^*} & 
    R\Gamma(\ol{\pi}^{-1}(\tau^\circ);\ul{\Q}) \ar[r]_-{i^*_{F_\tau}} \ar@{<--}@/^1pc/[r]^-{pr^*} & 
    R\Gamma(F_\tau;\ul{\Q}),
    }
    \]
    where $pr\colon \tau^\circ \times F_{\tau} \to F_\tau$ is the second projection, and $i_{F_\tau}\colon F_\tau \cong \{r(x)\}\times F_{\tau}\into \tau^\circ \times F_\tau$ is the inclusion of the fiber. It follows that the restriction map $\F_{\tau} \to \F_{\sigma}$ is equivalent to
    \[
    \xymatrix{
    R\Gamma(F_\tau;\ul\Q) \ar[r]^-{\tilde{r}^* (pr)^*} & 
    R\Gamma(\ol{\pi}^{-1}(\st(\tau));\ul{\Q}) \ar[r]^-{(i^{\tau}_{\sigma})^*} &
    R\Gamma(\ol{\pi}^{-1}(\st(\sigma);\ul\Q) \ar[r]^-{i_{F_\sigma}^*} & R\Gamma(F_\sigma;\ul\Q),
    }
    \]
whose composition is induced by $F_\sigma \cong \{x\}\times F_\sigma \into \ol{\pi}^{-1}(\st(\tau)) \overset{\tilde{r}}\to \tau^\circ\times F_\tau \to F_\tau$, as claimed.
\end{rmk}

\section{2-connected Outer Space and the universal 
$n$-marked graph}\label{sec:outer-space-and-friends}

All graphs shall be finite and connected, possibly with loops and parallel edges. The \emph{valence} of a vertex is the number of half-edges adjacent to it. The \emph{genus} of a graph is its first Betti number. If $\G$ is a connected graph, we say $\G$ is \emph{2-connected} if it remains connected, as a topological space, after removing any one vertex without removing its adjacent edges. Note that this is a slightly different definition than the classical one for simple graphs.\footnote{The subtlety here is that graphs 
may have loops, making them
not 2-connected. In the classical graph-theoretic definition one would also remove edges adjacent to the vertex, thereby erasing loops.}

Let $\Delta_{g}^{(2)}$ be the moduli space of $2$-connected graphs of genus $g$ and no bivalent vertices, equipped with a unit-volume metric. Similarly, $\Delta_{g,n}^{(2)}$ is the moduli space of such metric graphs further equipped with a choice of $n$ labeled vertices. Precise definitions are given below. We construct these spaces as quotients of ideal simplicial complexes with an action by the group $\Out(F_g)$, fitting into a pullback diagram
\begin{equation}\label{diagram: main diagram}
\begin{tikzcd}
\U_{g,n}^{(2)} \arrow[r,"\ol f"]\arrow[d,"\ol \pi"] & \Delta_{g,n}^{(2)} \arrow[d, "\pi"]\\
\mathrm{CV}_g^{(2)} \arrow[r, "f"] & \Delta_g^{(2)}
\end{tikzcd}
\end{equation}
to be used in the proof of Theorem~\ref{thm:ss}.
Here $\CV_g^{(2)}$ is the locus inside Culler-Vogtmann Outer Space (recalled below) of $2$-connected marked metric graphs, and $\U_{g,n}^{(2)}$ is an $n$-labeled analogue thereof.

Recall from Section~\ref{sec:graph complexes} the category $\GrphCat_g$ of graphs (finite connected multigraphs) of genus $g$, whose morphisms are compositions of isomorphisms and edge contractions. There is a geometric realization functor \[|\cdot|\col \GrphCat_g \to \Top\] that associates a topological space $|G|$ to a graph $G$.  We shall slightly abuse notation by writing $G$ both for the graph and for the topological space.

Let $\GrphCat_g^{(2)}$ denote the full subcategory of $\GrphCat_g$ whose objects are all $2$-connected graphs.
Consider a functor
\[\sigma^{(2)}=\sigma^{(2)}_g:\GrphCat^{(2),\mathrm{op}}_g \to \Top\]
that associates to a 2-connected graph $\G$ the space, denoted $\sigma^{(2)}_\G$, of unit volume length functions $l\col E(\G)\to \R_{\ge0}$ such that the contraction $\G/l^{-1}(0)$ is still a $2$-connected graph of genus $g$.  This space is a subspace  of the closed unit simplex of all unit volume length functions \[\overline{\sigma}_G = \{l\col E(\G)\to \R_{\ge0}: \sum l(e) = 1\}.\] In fact it is an ideal simplex, obtained from 
$\ol{\sigma}_G$
by removing closed faces. 

Then $\sigma^{(2)}$ is a functor: a morphism $\phi\col \G\to \G'$ determines an injection $E(\phi)\col E(\G')\rightarrow E(\G)$.  Write $T = E(\G) \setminus \im(E(\phi))$.  Then $\phi$ determines a map $\phi_*\col \overline{\sigma}_{\G'}\to \overline{\sigma}_\G$, which is an isomorphism onto the face of $\overline{\sigma}_\G$ parametrizing those $l\col E(\G)\to \R_{\ge0}$ that are $0$ on $T$.  Then $\phi_*$ restricts to a map \begin{equation}\label{eq:phi}\phi_*\col \sigma_{\G'}^{(2)} \rightarrow \sigma_{\G}^{(2)},\end{equation}
verified as follows: given $\ell'\in \sigma_{\G'}^{(2)}$, let $S = (\ell')^{-1}(0) \subset E(\G')$, so that $\G'/S$ is $2$-connected and genus $g$.  Setting $\ell = \phi_*(\ell')$ and identifying $E(\G')$ with a subset of $E(\G)$ via the injection $E(\phi)$, we have $\ell^{-1}(0) = S\cup T$, and thus $\G/(S\cup T) \cong \G'/S$ is also $2$-connected and of genus $g$. Hence $\ell \in \sigma_\G^{(2)}$, and $\phi_*(\sigma_{\G'}^{(2)})\subset \sigma_{\G}^{(2)}$ as desired.

As above,  denote the graph with one vertex and $g$ loops by $R_g$, called the \emph{rose graph}.

\subsection{2-connected $n$-marked Outer Space}\label{sec:2connected}
We define a space $\cU_{g,n}^{(2)}$ that parametrizes 2-connected graphs of genus $g$, equipped with $n$ labeled vertices and a homotopy class of homotopy equivalences from $R_g$. When $n=0$, the space $\cU_{g}^{(2)} = \cU_{g,0}^{(2)}$ is isomorphic, as an ideal simplicial complex, to the subspace of the Culler--Vogtmann Outer Space \cite{culler-vogtmann} parametrizing 2-connected graphs. 
As such, we will denote $\cU_{g,0}^{(2)}$ by $\CV_g^{(2)}$.
When $n>0$, the space $\cU_{g,n}^{(2)}$ is reminiscent of the Hatcher--Vogtmann Outer Space for thorned graphs (or graphs with leaves, see \cite{hatcher-vogtmann-homology-stability}), but the two spaces differ crucially in that points in our space $\cU_{g,n}^{(2)}$ are graphs with a homotopy marking by the rose $R_g$, and \emph{not} to the $n$-thorned rose.

Let $I_{g,n}^{(2)}$ be the category whose objects are triples $(\G,m,h)$, where
\begin{itemize}
\item $\G$ is a 2-connected multigraph,
\item $m\col \{1,\ldots,n\}\to V(\G)$ is any function, and
\item $h = [H]$ is a homotopy class of homotopy equivalences $H\col R_g\to \G$,
\end{itemize}
such that $m^{-1}(v) \ne \emptyset$ for any bivalent vertex $v\in V(\G)$. In words, these are $R_g$-marked, 2-connected graphs with $n$ labels  and no unlabeled bivalent vertices. A morphism $(\G,m,h) \to (\G',m',h')$ consists of a morphism of graphs (composition of isomorphisms and edge contractions) $j: \G \to \G'$ that satisfies:
\begin{itemize}
    \item $m' = V(j)\circ m$, where $V(j): V(\G)\to V(\G')$ is the induced map on vertices, and
    \item the diagram
\[\begin{tikzcd}
& \G \arrow[dd,"j"]\\
R_g \arrow[ru,"h"] \arrow[rd,"h'"'] & \\
& \G'
\end{tikzcd}\]    
    commutes in the homotopy category of topological spaces.
\end{itemize}

We also consider the category $\Gamma_{g,n}^{(2)}$ of pairs $(\G,m)$ subject to the same conditions, obtained by omitting $h$ from the definition of $I_{g,n}^{(2)}$ above. We call such a pair $(\G,m)$ a 2-connected \emph{$n$-marked\footnote{As in existing literature, the word \emph{marking} is used in two distinct contexts. A \emph{homotopy marking} of a graph is a homotopy equivalence with a rose graph, while an \emph{$n$-marking} is a labeling of $n$ (not necessarily distinct) vertices of the graph.} 
graph}.
Now, there are obvious faithful forgetful functors 
\[I^{(2)}_{g,n}
\xrightarrow{F_n} 
\Gamma^{(2)}_{g,n} 
\xrightarrow{P_n}
\GrphCat^{(2)}_g, \quad (\G,m,h) \mapsto (\G,m) \mapsto \G.\]
Pulling back the functor $\sigma^{(2)}$ along these lets us define the spaces we want.
\begin{defn}\label{def:Ugn and CVg}
    The \emph{2-connected, $n$-marked Outer Space} $\cU_{g,n}^{(2)}$ is the colimit of \[\sigma^{(2)} \circ P_n^\mathrm{op} \circ F_n^\mathrm{op}: I_{g,n}^{(2),\op} \to \Top.\]
    The  \emph{2-connected, $n$-marked moduli space of tropical curves} $\Delta_{g,n}^{(2)}$ is the colimit of \[\sigma^{(2)}\circ P_n^\op: \Gamma_{g,n}^{(2),\op}\to \Top.\] 
\end{defn}
    Explicitly, 
    \begin{equation} \label{eq:ugn2}
    \cU_{g,n}^{(2)} = \left( \coprod_{(\G,m,h)} \sigma_{\G}^{(2)} \right)\raisebox{-2ex}{$/\sim$}
    \end{equation}
    where $(\ell\col E(\G)\to \R_{\geq 0}) \sim (\ell'\col E(\G')\to \R_{\geq 0})$ if there exists a morphism $\phi\col (\G,m,h)\to (\G',m',h')$ such that $\phi_*(\ell') = \ell$ (Equation~\eqref{eq:phi}).  By suppressing mention of $h$ in~\eqref{eq:ugn2}, we also have an analogous explicit description of $\Delta_{g,n}^{(2)}$.

The universal property of colimits furnishes a natural map $f_n\colon \cU_{g,n}^{(2)} \to \Delta_{g,n}^{(2)}$.

\begin{rmk}
    When $n=0$ one might as well omit the labeling $m\colon [0]\to V(\G)$, and consider the categories of pairs $(\G,h)$ and $\G$, denoted $I^{(2)}_g$ and $\Gamma^{(2)}_g$. Accordingly, we omit the $n$ from the notation in the spaces defined using these categories.
    Recall that we denote $\CV_g^{(2)}:=\cU_{g,0}^{(2)}$ in this case, as it is isomorphic to the ideal simplicial subcomplex of Outer Space parametrizing 2-connected graphs.
\end{rmk}

\subsection{Repeated marking loci}
The spaces $\cU_{g,n}^{(2)}$ and $\Delta_{g,n}^{(2)}$ have subspaces parametrizing graphs in which some vertex is labeled by more than one element in $[n]$. 
Consider the full subcategory $I_{g,n}^{(2,r)}\subseteq I_{g,n}^{(2)}$ whose objects are triples $(\G,m,h)$ such that $m$ is not injective. Let $\cU_{g,n}^{(2,r)}$ be the colimit of the restriction to $I_{g,n}^{(2,r),\mathrm{op}}$  of \[\sigma^{(2)}\circ P_n^\op \circ F_n^\op.\] It naturally includes into $\cU_{g,n}^{(2)}$, and  and will be called the \emph{repeated marking locus}.  Note that $\cU_{g,n}^{(2,r)}$ is a closed subspace of $\cU_{g,n}^{(2)}$, since 
 edge contraction preserves the non-injectivity of a vertex labeling.

Similarly, let $\Gamma_{g,n}^{(2,r)}$ be the full subcategory of $\Gamma_{g,n}^{(2)}$ on objects $(G,m)$ where $m$ is not injective.  Then define $\Delta_{g,n}^{(2,r)}$ to be the colimit of the restriction to $\Gamma_{g,n}^{(2,r),\mathrm{op}}$ of $\sigma^{(2)}\circ P_n^{\mathrm{op}}$.  It is identified with a closed subspace $\Delta_{g,n}^{(2,r)} \subseteq \Delta_{g,n}^{(2)}$.

\subsection{Forgetting marked vertices}

Forgetting the $n$ marked vertices poses a slight complication, due to our requirement that bivalent vertices must be labeled. This is resolved by \emph{suppressing} such vertices: that is, deleting the bivalent vertex and connecting its two neighbors directly by an edge.

Define a functor
\begin{equation}\label{eq:S_n functor}
S_n:\Gamma_{g,n}^{(2)} \to \Gamma_{g}^{(2)}
\end{equation}
where $S_n(\G,m)$ is the graph $\G'$ obtained from $G$ by suppressing bivalent vertices. 
In particular, $\G$ is identified with a subdivision of $\G'$, and every edge $e\in E(\G)$ is obtained from subdividing a unique edge of $\G'$.
Then, given a function $\ell\colon E(\G)\to\R_{\geq0}$, define $\ell'\colon E(\G')\to\R_{\geq0}$
by $\ell'(e') = \sum\ell(e)$, where the sum is taken over edges $e\in E(\G)$ obtained from the subdivision of $e'$.
This defines a map on simplices $\sigma_{\G}^{(2)}\to\sigma_{\G'}^{(2)}$, giving the components of a natural transformation
\[\pi\colon \sigma^{(2)}\circ P_n^\op \Rightarrow \sigma^{(2)}\circ P_0^\op\circ S_n^\op\]
of functors $\Gamma_{g,n}^\op\to\Top$.
Such a natural transformation induces a map on the colimits of $\sigma^{(2)}\circ P_n^{\mathrm{op}}$ and $\sigma^{(2)}\circ P_0^{\mathrm{op}}$, i.e., the desired map $\Delta_{g,n}^{(2)}\to\Delta_g^{(2)}$.

Introducing a homotopy equivalence from $R_g$, there is an analogous functor $I_{g,n}^{(2)} \to I_g^{(2)}$ given by $F^*(S_n)\col (G,m,h)\mapsto (\G',h)$, where $\G'$ is determined by $\G$ as above. These give a commutative diagram\footnote{In fact, $F^*(S_n)$ is simply the pullback of $S_n$ along $F$, on the $2$-fiber product diagram~\eqref{eq:2-fiber-prod}.  
}
\begin{equation}\label{eq:2-fiber-prod}\begin{split}
\xymatrix{
I^{(2)}_{g,n} \ar[r]^{F_n} \ar[d]_{F^*(S_n)} & \Gamma_{g,n}^{(2)} \ar[d]^{S_n} \\
I^{(2)}_g \ar[r]^{F} & \Gamma_g^{(2)}.
}
\end{split}
\end{equation}
The natural transformations $\pi$ and $F^*(\pi)$ define a commutative square of functors and thus a commutative square of their colimits in topological spaces
\begin{equation} \begin{split}
\xymatrix{ \cU_{g,n}^{(2)} \ar[r]^{\ol{f}} \ar[d]_{\ol{\pi}} & \Delta_{g,n}^{(2)} \ar[d]^{\pi} \\ \CV_g^{(2)} \ar[r]^f & \Delta_g^{(2)}.}
\end{split}
\end{equation}

\subsection{$\Out(F_g)$-action}
Recall from \S\ref{sec:graph complexes} that  the group $\Out(F_g) \cong \operatorname{hEq}(R_g,R_g)$ acts naturally and simply transitively on the set of homotopy classes of homotopy equivalences $h\col R_g\to \G$.  Thus there are induced actions of $\Out(F_g)$ on the spaces 
$\cU_{g,n}^{(2)}$ and $\CV_g^{(2)}$, so that the map $\ol{\pi}$ between them is equivariant.  Precisely, these actions arise as follows: 
there is a natural action of $\Out(F_g) \cong \operatorname{hEq}(R_g,R_g)$ on the entire category $I_{g,n}^{(2)}$ and, compatibly, on $I_{g}^{(2)}$.   Formally, this means that $\Out(F_g)$ acts on the respective identity functors by natural transformations.  Functoriality of colimits gives the induced actions of $\Out(F_g)$ on  $\cU_{g,n}^{(2)}$ and $\CV_g^{(2)}$.

Since the $\Out(F_g)$-action on equivalences $\operatorname{hEq}(R_g,G)$ is simply transitive, the quotient of $I_{g,n}^{(2)}$ by $\Out(F_g)$ coincides with the category in which the equivalence $h:R_g\to \G$ is suppressed, and similarly for $I_{g}^{(2)}$. But these are exactly the categories $\Gamma_{g,n}^{(2)}$ and $\Gamma_{g}^{(2)}$, respectively, and the forgetful maps $F_n$ and $F$ are the quotient maps.

Since the quotient by a group action, being itself a colimit, commutes with colimits, the above maps thus induce homeomorphisms fitting into the diagram on the left, below.
\begin{equation}\label{diag:quotients by Out}
\begin{split}
    \xymatrix{
\cU_{g,n}^{(2)}/\!\Out(F_g) \ar[r]^-{\cong} \ar[d] & \Delta_{g,n}^{(2)} \ar[d] \\
\CV^{(2)}_g\!/\!\Out(F_g) \ar[r]^-{\cong} & \Delta_g^{(2)}
}\qquad\qquad \xymatrix{
\cU_{g,n}^{(2,r)}\!/\!\Out(F_g) \ar[r]^-{\cong} \ar[d] & \Delta_{g,n}^{(2,r)} \ar[d] \\
\CV^{(2)}_g\!/\!\Out(F_g) \ar[r]^-{\cong} & \Delta_g^{(2)}.
}
\end{split}
\end{equation}
The same constructions and identifications apply to the repeated marking loci, yielding the diagram on the right. It is compatible with the diagram on the left via inclusions.

\section{Proof of Theorem~\ref{thm:main}}\label{sec:proof} 
The proof strategy is now to apply Corollary \ref{cor:geometric cech bicomplex} to the diagram in \eqref{diagram: main diagram} and its repeated marking locus analogue. For this purpose, we now examine the fibers of the map $\cU_{g,n}^{(2)}\to\CV_g^{(2)}$
forgetting marked points.
Their resulting cohomologies, as well as the relative cohomology of the pair, will be computed by graph complexes, as defined in Section \ref{sec:graph complexes}, thus completing the proof of Theorem \ref{thm:main}.

\subsection{Fibers of the map forgetting marked vertices} 
We begin by observing that the restriction of the map 
\[\overline{\pi}\col \mathcal{U}_{g,n}^{(2)}\longrightarrow \mathrm{CV}_g^{(2)}\] over an open simplex 
$\sigma_{\G}^\circ$ of $\CV_g^{(2)}$ is a trivial $\G^n$-bundle.  More precisely, for each $(G,h)\in I_g^{(2)}$, 
the diagram below on the left is a pullback square:
    \begin{equation}\label{eq:trivial-over-open-simplex}
    \xymatrix{
    {\sigma}^{\circ}_{\G}\times \G^n \ar[d] \ar[r] & \U_{g,n}^{(2)} \ar[d]^{\overline{\pi}} \\
    {\sigma}^{\circ}_{\G} \ar[r]^{\iota_{(\G,h)}} & \CV_g^{(2)}
    }
    \qquad\quad
    \xymatrix{
    {\sigma}^{\circ}_{\G} \times \diag_n(\G) \ar[d] \ar[r] & \U_{g,n}^{(2,r)} \ar[d]^{\overline{\pi}^{(r)}} \\
    {\sigma}^{\circ}_{\G} \ar[r]^{\iota_{(\G,h)}} & \CV_g^{(2)}.
    }
    \end{equation}
Similarly, let 
\begin{equation} \label{eq:pi-repeating}\overline{\pi}^{(r)}\col \mathcal{U}_{g,n}^{(2,r)}\longrightarrow \mathrm{CV}_g^{(2)}\end{equation}
be the restriction of $\overline{\pi}$ to $\mathcal{U}_{g,n}^{(2,r)}$.
Then we have the pullback square on the right in~\eqref{eq:trivial-over-open-simplex}, where 
\begin{equation}\label{eq:the-diag}
\diag_n(\G) = \{(m_1,\dots,m_n)\in\G^n \mid m_i=m_j \text{ for some } i\neq j\}.    \end{equation}
Moreover, the two pullback squares are compatible with each other via inclusions.

It remains to relate fibers over different open simplices, which we do via 
retractions onto open simplices of their open stars.
\begin{prop}\label{prop:retraction on stars}
    For every $(\G,h)\in I_{g}^{(2)}$ there exist compatible 
    retractions $\tilde{r}$ and $r$ making the diagram below commute:
 \begin{equation}\label{eq:compatible-retractions}\begin{split}
       \xymatrix{
        \ol{\pi}^{-1}(\st(\sigma_{(\G,h)})) \ar[r]^{\tilde{r}} \ar[d]_{\ol{\pi}} & \ol{\pi}^{-1}(\sigma_{(\G,h)}^\circ) \ar[d]^{\ol{\pi}} \\
        \st(\sigma_{(\G,h)}) \ar[r]^{r} & \sigma_{(\G,h)}^{\circ} 
        }
    \end{split}
   \end{equation}
Moreover, for every morphism $\phi\colon(\G',h')\to(\G,h)$ in $I_g^{(2)}$, the restriction of $\tilde{r}$ to the preimage of the open simplex $\sigma^\circ_{(\G',h')}$ is naturally isomorphic to the map
\[
r\times \phi^n \colon \sigma_{(\G',h')}^\circ \times (\G')^n \longrightarrow \sigma_{(\G,h)}^\circ\times \G^n
\]
given as the product of the edge contraction $\G'\to \G$ with $r\colon \sigma_{(\G',h')}\to \sigma_{(\G,h)}$.

The analogous statement with repeated markings, namely replacing $\overline{\pi}$ with $\overline{\pi}^{(r)}$ and replacing $\G^n$ and $(\G')^n$ with $\diag_n(\G)$ and $\diag_n(\G')$ respectively, also holds.

\end{prop}
\begin{proof}
First, there is always a canonical retraction
\[r\col\st(\sigma_{(G,h)})\to\sigma^\circ_{(G,h)}\]
from the open star of a simplex, in this case $\sigma_{(G,h)}$, to the open interior of that simplex.  Here, this retraction can be described explicitly as follows: the open star $\st(\sigma_{(G,h)})$ is covered by its intersections with all of the closed simplices $\sigma_{(G',h')}$, for $(G',h')$ admitting an arrow to $(G,h)$ in the category $I_g^{(2)}$.  On such an intersection, the map 
\begin{equation}\label{eq:each-closed-simplex}\sigma_{(G',h')} \cap \st(\sigma_{(G,h)}) \to\sigma^\circ_{(G,h)}\end{equation}
projects onto the coordinates of $\sigma_{(G',h')}$ indexed by edges \emph{not} contracted by $\phi$, and then rescales all edge lengths to be unit volume.

Next, we define a retraction 
\[\tilde{r}\colon \ol{\pi}^{-1}(\st(\sigma_{(G,h)})) \to \ol{\pi}^{-1}(\sigma^\circ_{(G,h)})\]
which is compatible with $r$ in the sense that the diagram~\eqref{eq:compatible-retractions} commutes.
The preimage $\overline{\pi}^{-1}(\st(\sigma_{(G,h)}))$ is covered by its intersections with all closed simplices $\sigma_{(G',m',h')}$ where $(G',m',h')$ ranges over all objects of $I_{g,n}^{(2)}$ such that the ``forget marked points'' functor $F^*(S_n)$ (see \eqref{eq:2-fiber-prod}) sends $(G',m',h')$ to an object $(\ul{G}',\ul{h}')$ which in turn admits an arrow $(\ul{G}',\ul{h}')\xrightarrow{\phi}(G,h)$.  Described in reverse, $(G',m',h')$ may be obtained from $(G,h)$ by ``uncontracting'' some edges to obtain some $(\ul{G}',\ul{h}')$, and then introducing an $n$-marking.

Then the retraction $\tilde{r}$ is defined on the intersections
\[\sigma_{(G',m',h')} \cap \ol{\pi}^{-1}(\st(\sigma_{(G,h)}))\]
analogously to~\eqref{eq:each-closed-simplex}: by projecting to the coordinates of $\sigma_{(G',m',h')}$ indexed by exactly the edges of $G'$ which are subdivided pieces of edges in $\ul{G}'$ that are not contracted by $\phi$.  These glue together to the retraction $\tilde{r}$. The statement about the composition $\tilde{r}\tilde{\iota}$ follows from this explicit description of $\tilde{r}$.  

The corresponding statement with repeated markings has a similar proof.
\end{proof}

\begin{rmk}
In fact, $\cU_{g,n}^{(2)}$ and $\cU_{g,n}^{(2,r)}$ can alternatively be described as coends 
\[
\cU_{g,n}^{(2)} = \int^{(\G,h)} \sigma_\G \times \G^n \quad \text{ and similarly, } \quad \cU_{g,n}^{(2,r)} = \int^{(\G,h)} \sigma_\G \times \diag_n(\G)
\]
so that the map $\ol{\pi}$ is given by integrating the first projection onto $\int^{(\G,h)}\sigma_\G = \CV_g^{(2)}$.
\end{rmk}

\subsection{The tropical Serre spectral sequence}
The goal now is to obtain a double complex from which the spectral sequence of Theorem \ref{thm:main} arises. It will come from applying Corollary \ref{cor:geometric cech bicomplex} to the various spaces sitting above $\CV_g^{(2)}$. The trivializations ~\eqref{eq:trivial-over-open-simplex} over open simplices are exactly the input needed for the construction.

Our proof relies on understanding the \textit{non-repeated marking locus},
in which all marked vertices are distinct. It is realized as the complement of the repeated marking locus $\Delta_{g,n}^{(2,r)} \subseteq \Delta_{g,n}^{(2)}$, and so we gain access to its cohomology via the following.
\begin{lem}\label{lem:cech}
There is a commutative square 
\begin{center}
\begin{tikzpicture} \label{eq:cech for Delta repeated}
\node (r) at (0,0) {$C^*_c(\Delta_{g,n}^{(2,r)};\Q)$};
\node (2) at (0,3) {$C^*_c(\Delta_{g,n}^{(2)};\Q)$};
\node (G) at (8,3) {$\displaystyle\bigoplus_{[\G]\in\Iso(\Gamma_g^{(2)})}^{\phantom{{{X^X}^X}^X}} \left(C^*(\G^n;{\Q})\otimes\det(E(\G))[1]\right)^{\Aut(\G)}$};
\node (D) at (7.5,0) {$\displaystyle\bigoplus_{[\G]\in\Iso(\Gamma_g^{(2)})}^{\phantom{{{X^X}^X}^X}} \left(C^*(\diag_n(\G);{\Q})\otimes\det(E(\G))[1]\right)^{\Aut(\G)}$};
\draw[->] (2)--(r); 
\draw[->] (2) -- node[above]{$\simeq$} (G);
\draw[->] (r) -- node[above]{$\simeq$} (D);
\draw[->] (7.85,2.5) -- (7.85,.5);
\end{tikzpicture}
\end{center}
where the horizontal maps are quasi-isomorphisms and the vertical maps are induced by the inclusions $\Delta_{g,n}^{(2,r)}\into\Delta_{g,n}^{(2)}$ and $\diag_n(\G)\into \G^n$.

The terms on the right hand side are each a total complex of a double complex whose 
 $(p,q)$ component corresponds to $q$-cochains and graphs $\G$ with $p+1$ edges, and the differentials are given by edge contraction and the internal differential on each cochain complex.
\end{lem}

\begin{proof}
    We only prove the upper horizontal quasi-isomorphism. The argument for the lower horizontal map is analogous, and commutativity of the square is automatic by naturality of all constructions.

    We write $\Gp$ for $\Out(F_g)$ throughout. Applying Corollary \ref{cor:geometric cech bicomplex} to the map $\cU_{g,n}^{(2)} \to \CV_g^{(2)}$ gives a quasi-isomorphism with a \v{C}ech double-complex
    \[
    R\Gamma_c(\cU_{g,n}^{(2)}/\Gp ;\ul{\Q}) \simeq \bigoplus_{[\sigma_{(\G,h)}] \in P(\CV_g^{(2)})/\Gp} \left( R\Gamma( \G^n ; \ul{\Q} )\otimes \det{(\sigma_{(\G,h)})[1]}\right)^{\stab_\Gp(\sigma_{(\G,h)})}.
    \]
    Now, as explained in Diagram \eqref{diag:quotients by Out}, the quotient $\cU_{g,n}^{(2)}/\Gp$ is homeomorphic to $\Delta_{g,n}^{(2)}$, so the left-hand side of the last equivalence computes the cohomology we seek.

    Let us identify the right-hand side as a type of graph complex. Indeed, the ideal simplices in $\CV_g^{(2)}$ are in bijection with isomorphism classes of objects $(\G,h)\in I_{g}^{(2)}$, and the $\Gp$-action permutes them via its action on $h$. Since this action is transitive, orbits of isomorphism classes are in bijection with graphs $\G\in \Gamma_g^{(2)}$, with the marking $h$ forgotten. That is, $P(\CV_g^{(2)})/\Gp \cong \Iso(\Gamma_g^{(2)})$.
    
    Furthermore, since $g\geq 2$, the automorphism group $\Aut(\G)$ acts on the set of markings $R_g\to \G$ freely (it already acts faithfully on the homology of $\G$). Thus, by definition $I_g^{(2)}$, the $\Gp$-stabilizer of $[(\G,h)]\in \Iso(I_g^{(2)})$ is isomorphic to $\Aut(\G)$, naturally once we pick a representative $(\G,h)$. 
    Furthermore, $\stab_\Gp(\G,h)$ acts on the fiber $\G^n$ via $\Aut(\G)$.
    As for the determinant $\det(\sigma_{(\G,h)})$, the coordinates of the simplex are in bijection with edges $E(\G)$, compatibly under the action of $\Aut(\G)$. It therefore follows that the summands, groups and determinants are given as in the claimed presentation.

    Using the retraction on stars given in Proposition \ref{prop:retraction on stars}, 
    Remark \ref{rmk:Cech differential using retractions}
    identifies the \v{C}ech differentials with edge contraction on the space level.
    
    Lastly, the functor $\G \mapsto R\Gamma(\G^n;\ul{\Q})$ is naturally quasi-isomorphic to $\G \mapsto C^*(\G^n;\Q)$, e.g., 
    by \cite[Section 4]{petersen}. Therefore, we can replace the former for the latter in the \v{C}ech double-complex. Since $R\Gamma_c(\Delta_{g,n}^{(2)};\ul{\Q})$ is itself equivalent to $C^*_c(\Delta_{g,n}^{(2)})$, the latter is quasi-isomorphic to the claimed \v{C}ech double complex.

\end{proof}

Now, as promised, we use the preceding lemma to access the \emph{non-repeated marking locus} 
\[\Delta_g^{(2,nr)} := \Delta_g^{(2)}\setminus\Delta_g^{(2,r)}\]
and then complete the proof of Theorem \ref{thm:main}.

\begin{prop}[Double complex for $\Delta_{g,n}^{(2,nr)}$]\label{prop:our main double cx}
The compactly supported cohomology of the non-repeated marking locus $\Hc^*(\Delta_{g,n}^{(2,nr)};\Q)$ is computed by a double complex

\begin{equation} \label{eq:double complex - final form}
C^{p,q} = \bigoplus_{\substack{[\G]\in\Iso(\Gamma_g^{(2)})\\ |E(\G)|=p+1}} \left[C^q(\G^n,\diag_n(\G);{\Q})\otimes\det(E(\G))\right]^{\Aut(\G)}
\end{equation}
and whose differentials are given by edge contraction and the internal differential on each cochain complex.
\end{prop}

\begin{proof}
Both vertical arrows of the square \eqref{eq:cech for Delta repeated} are surjective (the right one since tensor products and invariants of a finite group on $\Q$-vector spaces are exact), and the $5$-lemma shows that the two respective kernels are also quasi-isomorphic.
The cohomology we seek is computed by the kernel of the left column, while the graph double-complex in the claim is the kernel of the right column. 
\end{proof}

\begin{proof}[Proof of Theorem \ref{thm:ss}]
The claim in our main theorem regards the cohomology of the moduli space of tropical curves $\Delta_{g,n}$, which parameterizes metric graphs of volume $1$ with vertex weights and marked vertices that satisfy certain genus and stability conditions. We will not define this space precisely, since \cite[Theorem 1.1]{cgp-marked} proves that
\[
\tilde{\Ho}^*(\Delta_{g,n};\Q) = \tilde{\Ho}^*_c(\Delta_{g,n};\Q) \cong \Hc^*(\Delta_{g,n}^{(2,nr)};\Q).
\]
Indeed, they show that the complement $\Delta_{g,n}\setminus \Delta_{g,n}^{(2,nr)}$ is contractible, and since it is furthermore compact, removing it does not alter the compactly supported cohomology. The first equality follows since $\Delta_{g,n}$ is compact.

It thus remains to exhibit the claimed spectral sequence computing $\Hc^*(\Delta_{g,n}^{(2,nr)};\Q)$. This will be the spectral sequence associated with the filtration of the double complex \eqref{eq:double complex - final form} by columns. Its $E_1$-page is
\[
E_1^{p,q} \cong \bigoplus_{\substack{[\G]\in\Iso(\Gamma_g^{(2)})\\ |E(\G)|=p+1}} \left[\Ho^q(\G^n,\diag_n(\G);{\Q})\otimes\det(E(\G))\right]^{\Aut(\G)}
\]

Excision in singular cohomology provides a natural isomorphism, for every finite graph $\G$,
\[
\Hc^*(\Conf_n(\G);\Q) \cong \Ho^*(\G^n,\diag_n(\G);\Q).
\]
Therefore, the terms of the $E_1$-page has the form claimed in the theorem.

In \cite[Proposition 2.1]{bibby-chan-gadish-yun-homology} we observe that $\Hc^*(\Conf_n(\G);\Q)$ is functorial and (proper) homotopy invariant. Since an edge contraction between finite $2$-connected graphs is a homotopy equivalence, it follows that the functor
\[
\G \mapsto \Hc^q(\Conf_n(\G);\Q)
\]
is a $\Gamma_g^{(2)}$-graph local system, and its associated 
graph complex is immediately seen to be the $q$-th row in the above $E_1$-page.

We furthermore prove in \cite[Lemma 2.4]{bibby-chan-gadish-yun-homology} that  $\Hc^*(\Conf_n(\G);\Q)$ is concentrated in degrees $n$ and $n-1$. This implies that the $E_1$-page is concentrated in rows $q=n$ and $q=n-1$, and the spectral sequence collapses on its $E_3$-page.
\end{proof}

\begin{cor}[Long exact sequence]\label{cor:les}
Let $\Gamma^{(2)}_g \leq \GrphCat_g$ denote the subcategory of $2$-connected graphs without bivalent vertices. The two graph complexes
\[
\GC^{*,q} := \GC^*(\Gamma^{(2)}_g; \Hc^q(\Conf_n(-);\Q))
\]
with $q=n$ and $n-1$ have cohomology $\GH^{*,q}$ that fits into a long exact sequence
\begin{equation}\label{eq:long exact sequence}
    \dots\to \GH^{i,n-1} \to \tilde\Ho^{i+n-1}(\Delta_{g,n};\Q) \to \GH^{i-1,n} \to
\GH^{i+1,n-1} \to \tilde\Ho^{i+n}(\Delta_{g,n};\Q) \to \dots
\end{equation}
\end{cor}
\begin{proof}
A spectral sequence with two nonzero rows is equivalent to a long exact sequence. Since the $E_1$-page of the sequence in Theorem \ref{thm:ss} is concentrated in rows $n$ and $n-1$, it gives rise to the claimed long exact sequence.
\end{proof}

\appendix

\section{Computational aspects}\label{appendix} 
The differentials in the graph complex can be better understood and simplified by accounting for the action of graphs' automorphism groups, as explained next.

Fix a full subcategory of graphs $\A$ satisfying the diamond property. Let us work with the presentation of the graph complex as a direct sum over isomorphism classes,
\[
GC(\A;V) \cong \bigoplus_{[\G]\in \A^{\cong}} \big(V_{\G}\otimes \det(E(\G))\big)^{\Aut(\G)}.
\]
As such, the graph differential breaks up into blocks
\[
d_{\G'}^\G: \big(V_{\G'}\otimes \det(E(\G'))\big)^{\Aut(\G')} \to \big(V_{\G}\otimes \det(E(\G))\big)^{\Aut(\G)}
\]
that are nonzero whenever $\G'\cong \G/e$ for some $e\in E(\G)$.

\begin{rmk}\label{rem:a-paper-differential}
$d^\G_{\G'}$ is given as follows, for fixed 
$\G$ and $\G'$ with $|E(\G)| = p+1 = |E(\G')|+1$. Choose, for each $e\in \G$ such that $\G/e\cong \G'$, a  particular morphism $\phi_e\col \G\to \G'$ that factors (indeed uniquely) as $\G\to \G/e\xrightarrow{\overline{\phi}_e} \G'$ for an isomorphism $\overline{\phi}_e$.  Then define \begin{equation}\label{eq:block differential}
\ol{d}_e = V_{\phi_e} \otimes (e\wedge \det \overline{\phi}_e).
\end{equation}
With this,  $d^\G_{\G'} = \sum_e \ol{d}_e$, summing over those edges $e$ for which $\G/e\cong \G'$.  At the moment, $d^\G_{\G'}$ is a map
\[V_{G'}\otimes \det E(\G') \to V_{\G} \otimes \det E(\G)\]
but a short verification shows that it induces a map, independent of all choices of $\phi_e$, 
from the $\Aut(\G')$-invariants of the domain to the $\Aut(\G)$-invariants of the codomain.
\end{rmk}

\begin{lem}
    Let $\G$ and $\G'$ be $\mathcal{A}$-graphs with $|E(\G)| = |E(\G')|+1$.  Let $E_{\G'}^\G \subseteq E(\G)$ be the edges $e$ such that $\G/e \cong \G'$. This set is $\Aut(\G)$-invariant, and let $[e_1],\dots,[e_k]$ be its set of orbits. 
    Then the differential $d_{\G'}^{\G}$ can be written in the following equivalent ways:
    \begin{equation}\label{eq:differenial}
    d_{\G'}^{\G} = \sum_{i=1}^k \left(
    \sum_{e'_i\in [e_i]} \sigma_{e'_i\to e_i}^*\right)\circ \ol{d}_{e_i} = \left(\frac{1}{|\Aut(\G)|} \sum_{\sigma\in \Aut(\G)} \sigma^*\right) \circ \sum_{i=1}^k |[e_i]| \cdot \ol{d}_{e_i}
    \end{equation}
    where the first internal sum goes over the elements in the orbit $[e_i]$ and  $\sigma_{e'_i\to e_i}\in \Aut(\G)$ is any choice of transporter $\sigma_{e'_i\to e_i}.e'_i=e_i$. Here, the map $\sigma^*$ is the automorphism of $V_\G\otimes\det(E(\G))$ induced by the action of $\sigma$ on $\G$, as in~\eqref{eq:intertwiner}, and $\ol{d}_e$ is defined in \eqref{eq:block differential}.
\end{lem}
Note that the map $\ol{d}_e$ depends on a choice of morphism $\phi_e:\G\to \G'$ contracting $e$. But since any two such choices differ by some automorphism of $\G'$, the map $\ol{d}_e$ is in fact independent of choices on the $\Aut(\G')$-invariants.
However, one should keep in mind that $\ol{d}_e$ is an isomorphism before restricting to invariants (albeit one that depends on choices).
\begin{proof}
    The set $E_{\G'}^\G$ is $\Aut(\G)$-invariant since an automorphism $\sigma$ of $\G$ induces an isomorphism $\ol{\sigma}: \G/e \isoto \G/\sigma(e)$.
    Grouping the edges in $E_{\G'}^\G$ by orbits, the differential has the form
    \[
    d_{\G'}^\G = \sum_{i=1}^k \sum_{e'_i\in [e_i]} \ol{d}_{e'_i} = \sum_{i=1}^k \sum_{e'_i\in [e_i]} \ol{d}_{\sigma_{e'_i\to e_i}^{-1}(e_i)}.
    \]
    We claim that for every $\sigma\in \Aut(\G)$ there exist a choice of $\ol{d}_{\sigma(e)}$ such that $\sigma^* \ol{d}_{\sigma(e)} = \ol{d}_{e}$.
    
    Indeed, the commutative diagram in $\GrphCat_g$
    \[
    \xymatrixcolsep{7pc}\xymatrix{
    \G' & \G/\sigma(e) \ar[l]_\sim & \G \ar[l] \\
     & \G/e \ar[ul]_\sim \ar[u]^{\ol{\sigma}}_\cong & \G \ar[u]^\sigma_\cong \ar[l]
     }
    \]
    induces a commutative diagram of vector spaces
    \[
    \xymatrix{
    V_{\G'}\otimes \det(E(\G')) \ar[dr]^-\sim \ar[r]^-\sim & V_{\G/\sigma(e)}\otimes \det(E(\G/\sigma(e))) \ar[d]^\cong_{\ol{\sigma}^*} \ar[r] &  V_{\G}\otimes \det(E(\G)) \ar[d]_{\sigma^*}^\cong \\
     & V_{\G/e}\otimes \det(E(\G/e)) \ar[r] & V_{\G}\otimes \det(E(\G)).
     }
    \]
    By definition, the composition of the bottom two arrows is the given choice for $\ol{d}_{e}$, while the composition of the top two is one for $\ol{d}_{\sigma(e)}$. From this, the first expression for $d_{\G'}^\G$ follows.

    For the second, observe the the sum over all $\Aut(\G)$ can be grouped by right $\stab(e_i)$-cosets
    \[
    \frac{1}{|\Aut(\G)|}\sum_{\sigma \in \Aut(\G)} \sigma^*  = \frac{1}{|[e_i]|} \sum_{[\sigma]\in \Aut(\G)/\stab(e_i)} \frac{1}{|\stab(e_i)|} \sum_{\tau\in \stab(e_i)} (\tau\sigma)^*.
    \]
    Composing, $(\tau\sigma)^*\ol{d}_{e_i} = \ol{d}_{\sigma^{-1}\tau^{-1}(e_i)} = \ol{d}_{\sigma^{-1}(e_i)}$ when considered as maps from $\Aut(\G')$-invariants, and where the second equality uses $\tau(e_i)=e_i$. We therefore deduce that for every $1\leq i\leq k$,
    \[
    \frac{1}{|\Aut(\G)|}\sum_{\sigma \in \Aut(\G)} \sigma^*\circ \ol{d}_{e_i} = \frac{1}{|[e_i]|} \sum_{[\sigma]\in \Aut(\G)/\stab(e_i)} \ol{d}_{\sigma^{-1}e_i}
    \]
    from which the second claimed description of $d_{\G'}^\G$ immediately follows.
\end{proof}

There are several scenarios in which terms in graph complex are guaranteed to not contribute to the cohomology. These are particularly fortuitous in the genus $3$ graph complex.

\begin{prop}[Differential of special edge]\label{prop:differential special edge}
Suppose that $\G$ and $\G'$ are genus $g$ graphs, and $e\in E(\G)$ is an edge such that $\G/e\cong \G'$.
Then there are injective group homomorphisms
\begin{center}
\begin{tikzpicture}
\node (G) at (-1.5,1.5) {$\Aut(\G)$};
\node (H) at (1.5,1.5) {$\Aut(\G').$};
\node (stab) at (0,3) {$\stab_{\Aut(\G)}(e)$};
\draw[left hook->] (stab) -- node[above,pos=0.7] {\footnotesize $i_e$} (G);
\draw[right hook->] (stab) -- node[above,pos=0.7] {\footnotesize $h_e$} (H);
\end{tikzpicture}
\end{center}
Further assume that for every other $e'\in E(\G)$, if $\G/e'\cong \G/e$ then there exists an automorphism $\varphi\in \Aut(\G)$ such that $\varphi(e')=e$.
\begin{enumerate}
\item\label{lemma-case:injective} If $i_e$ is surjective, i.e., $\stab_{\Aut(\G)}(e)=\Aut(\G)$, then $d_{\G'}^{\G}$ is injective.
\item\label{lemma-case:surjective} If $h_e$ is surjective and $\stab_{\Aut(\G)}(e)$ is a normal subgroup of $\Aut(\G)$, then $d_{\G'}^{\G}$ is surjective.
\end{enumerate}
\end{prop}

\begin{proof}
On the graph complex, the differential $d_{\G'}^{\G}$ fits into a commutative diagram
$$
\xymatrixcolsep{8pc}\xymatrixrowsep{3pc}\xymatrix{
V_{\G'}\otimes \det(E(\G'))^{\Aut(\G')} \ar[r]^{d_{\G'}^\G} \ar@{^(->}[d]_{\ol{d}_e} & V_{\G}\otimes \det(E(\G))^{\Aut(\G)} \ar@{=}[d] \\
V_{\G}\otimes \det(E(\G))^{\stab(e)} \ar[r]^{\displaystyle{\sum_{[\sigma]\in \Aut(\G)/\stab(e)} \sigma^*}} & V_{\G}\otimes \det(E(\G))^{\Aut(\G)}.
}
$$

In case \eqref{lemma-case:injective}, the hypothesis implies that the bottom arrow can be taken to be the identity. Therefore the differential is simply the restriction of the isomorphism $\ol{d}_e$ to $\Aut(\G')$-invariants.

In case \eqref{lemma-case:surjective}, the quotient group $Q = \Aut(\G)/\stab(e)$ acts on the invariants
\[
(V\otimes \det(E(\G)))^{\stab(e)} \text{ and the $Q$-invariants are exactly } (V\otimes \det(E(\G)))^{\Aut(\G)}.
\]
In particular, the differential $d_{\G'}^{\G}$ is the composition of the isomorphism $\ol{d}_e$ with a scalar multiple of the projector onto the $Q$-invariants, and is thus surjective.
\end{proof}

Knowing these facts about the blocks of the graph differential, one may prune the graph complex and reduce the number of graphs that may contribute to cohomology.
We applied this to the case of genus $g=3$ as the following two examples detail.

\begin{exmp}\label{ex:genus3 injective diff}
Recall the graph complexes of interest here are those associated to $\A=\Gamma_3^{(2)}$ and $\Out(F_3)$-representation $V=\Hc^q(\Conf_n(R_3);\Q))$ where $q=n$ or $q=n-1$.
Consider the graphs $G=\goggles$ and $G'=\banana$ depicted in Figure \ref{fig:banana-goggles}, with the only contractible edge $e$ of $\goggles$ indicated.
Every automorphism of $\goggles$ stabilizes the edge $e$.
Thus, by Proposition \ref{prop:differential special edge}\eqref{lemma-case:injective}, the graph complex differential $d_{\banana}^{\goggles}$ is injective.

It follows that the graph cohomology is concentrated in degrees $p=4$ and $p=5$, and \eqref{eq:long exact sequence} breaks up into
\begin{align}
    \GH^{4,n-1} \cong &H^{n+3}(\Delta_{3,n};\Q), \nonumber\\ 
    0 \to \GH^{5,n-1} \to &H^{n+4}(\Delta_{3,n};\Q) \to \GH^{4,n} \to 0 \text{, and} \label{eq:GH5n}\\
    &H^{n+5}(\Delta_{3,n};\Q) \cong \GH^{5,n}. \nonumber
\end{align}
\begin{figure}[ht]
\begin{tikzpicture}
\node at (0,-.15) {};
\tikzstyle{every node}=
[draw,circle,fill=black,minimum size=2pt,inner sep=0pt];
\node (0) at (0,0) {};
\node (1) at (0,.8) {};
\draw[-] (0) edge[bend right=85, min distance=5mm] (1);
\draw[-] (0) edge[bend right=45] (1);
\draw[-] (0) edge[bend left=45] (1);
\draw[-] (0) edge[bend left=85, min distance=5mm] (1);
\end{tikzpicture}
\qquad
\begin{tikzpicture}
\tikzstyle{every node}=
[draw,circle,fill=black,minimum size=2pt,inner sep=0pt];
\node (l) at (0,0) {};
\node (c) at (0.5,.8) {};
\node (r) at (1,0) {};
\tikzstyle{every node}=[scale=.7];
\draw[-] (l)-- node[below] {$e$} (r);
\draw[-] (l) edge[bend right=25] (c);
\draw[-] (r) edge[bend right=25] (c);
\draw[-] (l) edge[bend left=25] (c);
\draw[-] (r) edge[bend left=25] (c);
\end{tikzpicture}
\caption{A pair of graphs and edge $e$ satisfying hypotheses of Proposition \ref{prop:differential special edge}\eqref{lemma-case:injective}. See Example \ref{ex:genus3 injective diff}.}
\label{fig:banana-goggles}
\end{figure}
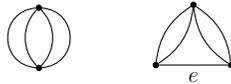

\end{exmp}

\begin{exmp}\label{ex:genus3 surjective diff}
Consider the graphs $G=\can$ and $G'=\goggles$ depicted in Figure \ref{fig:can-goggles}, with one of the two contractible edges of $\can$ labeled by $e$.
The only automorphism of $\can$ which does not stabilize $e$ is a (vertical) reflection, thus the stabilizer subgroup $\stab(e)$  has index 2 in $\Aut(\can)$, hence is a normal subgroup. 
Furthermore, every automorphism of $\goggles$ can be obtained from an automorphism of $\can$ stabilizing $e$ through the edge contraction. 
Thus, by Proposition \ref{prop:differential special edge}\eqref{lemma-case:surjective}, the graph complex differential $d_{\goggles}^{\can}$ is surjective.

This surjectivity allows us to prune the graph complex, simplifying the calculation of $H^{n+5}(\Delta_{3,n};\Q)$ as in \eqref{eq:GH5n}. 
That is, instead of computing the cokernel of the graph complex differential
\[d_{\goggles}^{\can}\oplus d_{\goggles}^{\kfour} :
(V\otimes \det(E(\goggles)))^{\Aut(\goggles)} \to
(V\otimes \det(E(\can)))^{\Aut(\can)} \oplus
(V\otimes \det(E(\kfour)))^{\Aut(\kfour)}
\]
one can compute the cokernel of a \emph{modified} differential
\[d' :
(V\otimes \det(E(\goggles)))^{\Aut(\goggles)} \to
(V\otimes \det(E(\kfour)))^{\Aut(\kfour)}.
\]
In particular, this modified differential can be represented by a significantly smaller matrix than the original one.
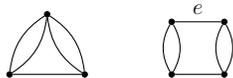
\begin{figure}[ht]\begin{tikzpicture}
\tikzstyle{every node}=
[draw,circle,fill=black,minimum size=2pt,inner sep=0pt];
\node (l) at (0,0) {};
\node (c) at (0.5,0.8) {};
\node (r) at (1,0) {};
\tikzstyle{every node}=[scale=.7];
\draw[-] (l) -- (r);
\draw[-] (l) edge[bend right=25] (c);
\draw[-] (r) edge[bend right=25] (c);
\draw[-] (l) edge[bend left=25] (c);
\draw[-] (r) edge[bend left=25] (c);
\end{tikzpicture}
\qquad
\begin{tikzpicture}
\tikzstyle{every node}=
[draw,circle,fill=black,minimum size=2pt,inner sep=0pt];
\node (bl) at (0,0) {};
\node (br) at (0,.7) {};
\node (tl) at (.7,0) {};
\node (tr) at (.7,.7) {};
\tikzstyle{every node}=[scale=.7];
\draw[-] (tl) edge[bend right=30] (tr);
\draw[-] (tl) edge[bend left=30] (tr);
\draw[-] (bl) edge[bend right=30] (br);
\draw[-] (bl) edge[bend left=30] (br);
\draw[-] (bl) edge (tl);
\draw[-] (br) edge node[above] {$e$} (tr);
\end{tikzpicture}
\caption{A pair of graphs and edge $e$ satisfying hypotheses of Proposition \ref{prop:differential special edge}\eqref{lemma-case:surjective}. See Example \ref{ex:genus3 surjective diff}.}
\label{fig:can-goggles}
\end{figure}

\end{exmp}

The outcome of applying the previous examples to generate data is discussed in the introduction; see Table \ref{table:genus3}. Our method allows one to compute individual multiplicities of irreducible representations: the multiplicities of the sign and trivial representations are displayed in Table \ref{table:sign_trivial} for $n\leq 13$.
These particular representations appear in the homology of hairy graph complexes, for which partial data was obtained by Khoroshkin--Willwacher--\v{Z}ivkovi\'c; see column 3 in Figures 1 (left) and 2 (left) of \cite{KWZ}.  They are consistent with Arone-Turchin's calculations of Euler characteristics of the hairy graph complexes $\mathrm{HGC}_{m,n}$ for $m$ odd and $n$ even and $m$ and $n$ both even \cite[Table A.3 and Table A.7]{arone-turchin-graph-complexes}.  To be clear, Euler characteristics do not provide a significant consistency check for our calculations, since our calculations took information on Euler characteristics---more precisely, the formula for the top-weight Euler characteristic of $\cM_{g,n}$~\cite{cfgp-sn}---as input.  So we are verifying that the two Euler characteristic computations in \cite{arone-turchin-graph-complexes} and in \cite{cfgp-sn} agree with each other. 
On the other hand, V.~Turchin observed to us that in the data in Table~\ref{table:sign_trivial} for $n\le 13$, all nonzero homology groups appear in equal parity; that is, the total homology rank equals the unsigned Euler characteristic. We do not know whether this statement is true for all $n$ The corresponding statement for genus $2$ {\em does} hold for all $n$, see 
the discussion and further references in \cite[Remark 3.4]{bibby-chan-gadish-yun-homology}.  On the other hand, it is evident from Table~\ref{table:genus3} that there are partitions $\lambda\vdash n$ other than $(n)$ and $(1^n)$ for which cancellation does occur in the calculation of the $\lambda$-isotypical component of the Euler characteristic.

\begin{table}[hbt]
{\footnotesize
\begin{tabular}{|c|r|r|r|}
\cline{2-4}
\multicolumn{1}{c|}{
\begin{tikzpicture}[scale=.2,baseline=1pt]
\node at (-1,0) {\footnotesize$n$};
\node at (1,1) {\footnotesize$i$};
\draw[-] (-3,2) -- (2,-1);
\end{tikzpicture}}
 & 0 & 1 & 2 \\
\cline{2-4}
\hline
$1$ & 0 & 1 & 0 \\
\hline
$2$ & 0 & 0 & 0 \\
\hline
$3$ & 1 & 0 & 0 \\
\hline
$4$ & 0 & 0 & 0 \\
\hline
$5$ & 0 & 1 & 0 \\
\hline
$6$ & 0 & 1 & 0 \\
\hline
$7$ & 1 & 0 & 0 \\
\hline
$8$ & 1 & 0 & 1 \\
\hline
$9$ & 0 & 1 & 0 \\
\hline
$10$ & 0 & 3 & 0 \\
\hline
$11$ & 2 & 0 & 0 \\
\hline
$12$ & 2 & 0 & 2 \\
\hline
$13$ & 0 & 3 & 0 \\
\hline
\end{tabular}
\qquad
\begin{tabular}{|c|r|r|r|}
\cline{2-4}
\multicolumn{1}{c|}{
\begin{tikzpicture}[scale=.2,baseline=1pt]
\node at (-1,0) {\footnotesize$n$};
\node at (1,1) {\footnotesize$i$};
\draw[-] (-3,2) -- (2,-1);
\end{tikzpicture}}
 & 0 & 1 & 2 \\
\cline{2-4}
\hline
$1$ & 0 & 1 & 0 \\
\hline
$2$ & 0 & 0 & 0 \\
\hline
$3$ & 0 & 0 & 0 \\
\hline
$4$ & 0 & 0 & 0 \\
\hline
$5$ & 0 & 0 & 0 \\
\hline
$6$ & 0 & 0 & 0 \\
\hline
$7$ & 0 & 0 & 0 \\
\hline
$8$ & 0 & 1 & 0 \\
\hline
$9$ & 1 & 0 & 1 \\
\hline
$10$ & 0 & 3 & 0 \\
\hline
$11$ & 1 & 0 & 2 \\
\hline
$12$ & 0 & 3 & 0 \\
\hline
$13$ & 2 & 0 & 2 \\
\hline
\end{tabular}
}
\bigskip

\caption{Multiplicities of trivial (left) and sign (right) representations in $\rH^{n+5-i}(\Delta_{3,n};\Q)$. Equivalently, by Proposition~\ref{prop:three-comparisons}: dimensions of homology groups for $\mathrm{HGC}_{m,N}$ for $m$ odd and $N$ even (left) and for $m$ and $N$ both even (right).}
\label{table:sign_trivial}
\end{table}


\begin{thebibliography}{CCTW14}

\bibitem[ACP15]{acp}
Dan Abramovich, Lucia Caporaso, and Sam Payne.
\newblock The tropicalization of the moduli space of curves.
\newblock {\em Ann. Sci. \'{E}c. Norm. Sup\'{e}r. (4)}, 48(4):765--809, 2015.

\bibitem[AT15]{arone-turchin-graph-complexes}
Gregory Arone and Victor Turchin.
\newblock Graph-complexes computing the rational homotopy of high dimensional
  analogues of spaces of long knots.
\newblock {\em Ann. Inst. Fourier (Grenoble)}, 65(1):1--62, 2015.

\bibitem[BBP22]{brendle-broaddus-putman}
T.~Brendle, N.~Broaddus, and A.~Putman.
\newblock The high-dimensional cohomology of the moduli space of curves with
  level structures {II}: punctures and boundary.
\newblock preprint arXiv:2003.10913, 2022.

\bibitem[BCGY23]{bibby-chan-gadish-yun-homology}
Christin Bibby, Melody Chan, Nir Gadish, and Claudia~He Yun.
\newblock Homology representations of compactified configurations on graphs
  applied to $\mathcal{M}_{2,n}$.
\newblock {\em Experimental Mathematics}, 2023.
\newblock To appear.

\bibitem[BCK23]{brandt-chan-kannan-on}
Madeline Brandt, Melody Chan, and Siddarth Kannan.
\newblock On the weight 0 compactly supported cohomology of
  $\mathcal{H}_{g,n}$.
\newblock 2023.

\bibitem[BV23]{borinsky-vermaseren-sn}
Michael Borinsky and Jos Vermaseren.
\newblock The ${S}_n$-equivariant {E}uler characteristic of the moduli space of
  graphs, 2023.

\bibitem[CCTW14]{conant-costello-turchin-weed-two}
Jim Conant, Jean Costello, Victor Turchin, and Patrick Weed.
\newblock Two-loop part of the rational homotopy of spaces of long embeddings.
\newblock {\em J. Knot Theory Ramifications}, 23(4):1450018, 23, 2014.

\bibitem[CFGP]{cfgp-sn}
Melody Chan, Carel Faber, S{\o}ren Galatius, and Sam Payne.
\newblock The ${S}_n$-equivariant top weight {E}uler characteristic of
  $\mathcal{M}_{g,n}$.
\newblock {\em Amer. J. Math.}
\newblock To appear.

\bibitem[CGP21]{cgp-graph-homology}
Melody Chan, S{\o}ren Galatius, and Sam Payne.
\newblock Tropical curves, graph complexes, and top weight cohomology of
  $\mathcal{M}_g$.
\newblock {\em J.~Amer.~Math.~Soc.}, 34:565--594, 2021.

\bibitem[CGP22]{cgp-marked}
Melody Chan, S{\o}ren Galatius, and Sam Payne.
\newblock Topology of moduli spaces of tropical curves with marked points.
\newblock In {\em Facets of algebraic geometry. {V}ol. {I}}, volume 472 of {\em
  London Math. Soc. Lecture Note Ser.}, pages 77--131. Cambridge Univ. Press,
  Cambridge, 2022.

\bibitem[CGV05]{conant-gerlits-vogtmann-cut}
James Conant, Ferenc Gerlits, and Karen Vogtmann.
\newblock Cut vertices in commutative graphs.
\newblock {\em Q. J. Math.}, 56(3):321--336, 2005.

\bibitem[CHKV16]{conant-hatcher-kassabov-vogtmann-assembling}
James Conant, Allen Hatcher, Martin Kassabov, and Karen Vogtmann.
\newblock Assembling homology classes in automorphism groups of free groups.
\newblock {\em Comment. Math. Helv.}, 91(4):751--806, 2016.

\bibitem[CHMR16]{cavalieri-hampe-markwig-ranganathan}
Renzo Cavalieri, Simon Hampe, Hannah Markwig, and Dhruv Ranganathan.
\newblock Moduli spaces of rational weighted stable curves and tropical
  geometry.
\newblock {\em Forum Math. Sigma}, 4:e9, 35, 2016.

\bibitem[CV86]{culler-vogtmann}
Marc Culler and Karen Vogtmann.
\newblock Moduli of graphs and automorphisms of free groups.
\newblock {\em Invent. Math.}, 84(1):91--119, 1986.

\bibitem[CV03]{conant-vogtmann-on}
James Conant and Karen Vogtmann.
\newblock On a theorem of {K}ontsevich.
\newblock {\em Algebr. Geom. Topol.}, 3:1167--1224, 2003.

\bibitem[GH22]{gadish-hainaut}
Nir Gadish and Louis Hainaut.
\newblock Configuration spaces on a wedge of spheres and hochschild-pirashvili
  homology.
\newblock {\em arXiv preprint arXiv:2202.12494}, 2022.

\bibitem[HP23]{hainaut-petersen}
Louis Hainaut and Dan Petersen.
\newblock Top weight cohomology of moduli spaces of {R}iemann surfaces and
  handlebodies.
\newblock arXiv:2305.03046, 2023.

\bibitem[HV04]{hatcher-vogtmann-homology-stability}
Allen Hatcher and Karen Vogtmann.
\newblock Homology stability for outer automorphism groups of free groups.
\newblock {\em Algebr. Geom. Topol.}, 4:1253--1272, 2004.

\bibitem[KLSY20]{kannan-li-serpente-yun-topology}
Siddarth Kannan, Shiyue Li, Stefano Serpente, and Claudia~He Yun.
\newblock Topology of tropical moduli spaces of weighted stable curves in
  higher genus.
\newblock arXiv:2010.11767, 2020.

\bibitem[KS94]{kashiwara-schapira}
Masaki Kashiwara and Pierre Schapira.
\newblock {\em Sheaves on manifolds}, volume 292 of {\em Grundlehren der
  mathematischen Wissenschaften [Fundamental Principles of Mathematical
  Sciences]}.
\newblock Springer-Verlag, Berlin, 1994.
\newblock With a chapter in French by Christian Houzel, Corrected reprint of
  the 1990 original.

\bibitem[KWv17]{KWZ}
Anton Khoroshkin, Thomas Willwacher, and Marko \v{Z}ivkovi\'{c}.
\newblock Differentials on graph complexes {II}: hairy graphs.
\newblock {\em Lett. Math. Phys.}, 107(10):1781--1797, 2017.

\bibitem[LV08]{LV}
A.~Lazarev and A.~A. Voronov.
\newblock Graph homology: {K}oszul and {V}erdier duality.
\newblock {\em Adv. Math.}, 218(6):1878--1894, 2008.

\bibitem[Pet22]{petersen}
Dan Petersen.
\newblock A remark on singular cohomology and sheaf cohomology.
\newblock {\em Math. Scand.}, 128(2):229--238, 2022.

\bibitem[STT18]{tsopmene-turchin-euler}
Paul~Arnaud Songhafouo~Tsopm\'{e}n\'{e} and Victor Turchin.
\newblock Euler characteristics for spaces of string links and the modular
  envelope of {$\mathcal{L}_\infty$}.
\newblock {\em Homology Homotopy Appl.}, 20(2):115--144, 2018.

\bibitem[TW17]{turchin-willwacher-commutative}
Victor Turchin and Thomas Willwacher.
\newblock Commutative hairy graphs and representations of
  $\mathrm{Out}({F}_r)$.
\newblock {\em Journal of Topology}, 10(2):386--411, 2017.

\bibitem[TW19]{turchin-willwacher-hochschild}
Victor Turchin and Thomas Willwacher.
\newblock Hochschild-{P}irashvili homology on suspensions and representations
  of {${\rm Out}(F_n)$}.
\newblock {\em Ann. Sci. \'{E}c. Norm. Sup\'{e}r. (4)}, 52(3):761--795, 2019.

\bibitem[Uli15]{ulirsch-tropical}
Martin Ulirsch.
\newblock Tropical geometry of moduli spaces of weighted stable curves.
\newblock {\em J. Lond. Math. Soc. (2)}, 92(2):427--450, 2015.

\end{thebibliography}
\end{document}